\documentclass{article}
\usepackage{amsthm}
\usepackage{amssymb}
\usepackage{amsmath}
\usepackage{bbold}
\usepackage{dsfont}
\usepackage{mathrsfs}
\usepackage{xspace}
\usepackage{enumerate}
\usepackage[pdftex]{graphicx}
\usepackage{subcaption}
\usepackage[format = hang]{caption}
\usepackage{subcaption}
\usepackage{array}
\usepackage{thmtools, thm-restate}
\usepackage[section]{placeins}
\usepackage{adjustbox}
\usepackage{placeins}

\newcommand{\given}{\bigm|}

\newcommand{\natsz}{\ensuremath{\mathds{N}_0}\xspace}

\newcommand{\CM}{\texttt{Cle\-ver\-Ma\-ker}\xspace}
\newcommand{\RM}{\texttt{Ran\-dom\-Ma\-ker}\xspace}
\newcommand{\CB}{\texttt{Cle\-ver\-Br\-eaker}\xspace}
\newcommand{\RB}{\texttt{Ran\-dom\-Br\-eaker}\xspace}

\newcommand{\size}[1]{\left|#1\right|}
\newcommand{\prob}[1]{ Pr \left[#1\right]}
\newcommand{\Hgraph}[1]{ H\left(#1\right)}
\newcommand{\HsigmaS}[1]{ H_{\sigma,S}\left(#1\right)}
\newcommand{\G}[1]{ G\left(#1\right)}
\newcommand{\Gsigma}[2]{ G_{#1}\left(#2\right)}

\declaretheorem[numberwithin=section, name=Theorem]{thm}
\declaretheorem[sibling=thm, name=Lemma]{lem}

\declaretheorem[sibling=thm, name=Proposition]{prop}
\declaretheorem[sibling=thm, name=Corollary]{cor}
\declaretheorem[sibling=thm, style=definition, name=Definition]{defin}

\title{Sharp thresholds for half-random games I}
\date{}
\author{Jonas Groschwitz\\ Universit\"at Potsdam  \and Tibor Szab\'o \\ Freie Universit\"at Berlin}

\begin{document}

\maketitle

\begin{abstract}
We study biased Maker-Breaker positional games between two players, one of whom is playing 
randomly against  an opponent with an optimal strategy. 
In this paper we consider the scenario when Maker plays randomly and
Breaker is ``clever'', and determine the sharp threshold bias of classical graph
games, such as connectivity, Hamiltonicity, and minimum degree-$k$. 
We treat the other case, that is when Breaker plays randomly, in a separate paper.

The traditional, deterministic version of these games, with two optimal players playing,
are known to obey the so-called probabilistic intuition. That is,
the threshold bias of these games is asymptotically equal to the threshold bias of their random counterpart, where players just 
take edges uniformly at random.
We find, that despite this remarkably precise agreement of the results of the deterministic and the random games, 
playing randomly against an optimal opponent is not a good idea: the
threshold bias becomes significantly more tilted towards the random player. 
An important qualitative aspect of the probabilistic intuition carries through nevertheless: the bottleneck for Maker to occupy a connected graph 
is still the ability to avoid isolated vertices in her graph.
\end{abstract}

\section{Introduction} 
Let us be given a finite hypergraph ${\cal F} \subseteq 2^X$ on a vertex set $X$.
In the Maker-Breaker positional game ${\cal F}$ two players, Maker and 
Breaker, alternately take turns in occupying free elements of $X$,
with Maker going first, until no free element is left.
Maker is the winner if he completely occupied a hyperedge of the hypergraph ${\cal F}$, 
otherwise Breaker wins. 
Such a game is of perfect information with no chance moves, so one of
the players has a winning strategy. That which one, depends on the hypergraph ${\cal F}$. 
A standard method, introduced by Chv\'atal and Erd\H os~\cite{chvaterd78}, to measure 
the robustness of this winning strategy 
is to give the ``disadvantaged'' player a {\em bias}, that is to allow him 
to occupy more than one element of $X$ in each turn. 
In an $(a:b)$ {\em biased game} Maker occupies $a$ elements of $X$ in each turn 
and Breaker occupies $b$ elements. 

For our investigation we will be concerned mostly with graph games, where the 
{\em board} $X$ is the edge set $E(K_n)$ of the complete graph and the game hypergraph
${\cal P} \subseteq 2^{E(K_n)}$ describes a graph property. 
In the present paper we study properties fundamental both in terms of
graph theory and  positional games. These include connectivity, having a perfect matching, Hamilton cycle or 
minimum degree $k$. For this, let ${\cal C}(n)$, ${\cal H}(n)$, ${\cal D}_k(n)$ denote the family of edge sets of $n$-vertex graphs
that are connected, contain a Hamiltonian cycle, have minimum
degree $k$, respectively. Subsequently we suppress the parameter $n$ in the notation.

\subsection{Threshold bias and probabilistic intuition}
As it turns out, many of the natural graph games are relatively easy wins for Maker
if the game is played $(1:1)$. 
Chv\'{a}tal and Erd\H{o}s~\cite{chvaterd78} were the first to study 
how large of a bias $b$ Breaker needs in various graph games in order to win
the $(1:b)$ biased game.
For a game hypergraph ${\cal F}$ we define $b_{\cal F}$ to be the smallest integer $b$ 
such that Breaker has a winning strategy in the $(1:b)$ biased game ${\cal F}$ and 
$b_{\cal F}$ is called the \textit{threshold bias} of the game. 

Chv\'{a}tal and Erd\H{o}s \cite{chvaterd78} determined the order of magnitude of the 
threshold bias of the connectivity game ${\cal C}$ and the triangle building game 
${\cal K}_{K_3}$. They have shown that $b_{\cal C} = \Theta \left(\frac{n}{\ln n}\right)$ 
and $b_{{\cal K}_{K_3}} = \Theta \left( \sqrt{n}\right)$. The constant factor
in the lower bound for $b_{\cal C}$ was first improved by Beck~\cite{beck82}.
Later Gebauer and Szab\'{o} \cite{gebsza09} established 
$b_{\cal C} = (1+o(1)) \frac{n}{\ln n}$, showing that the upper bound of Chv\'atal and 
Erd\H os is asymptotically tight.
For the Hamiltonicity game ${\cal H}$ Chv\'atal and Erd\H os only showed that 
$b_{\cal H} > 1$. This was subsequently improved in a series of papers by 
Bollob\'as and Papaioannou~\cite{BollobasPapa},
Beck~\cite{BeckHamilton}, Krivelevich and Szab\'o~\cite{SZKHamilton}, 
until Krivelevich \cite{kriv09} proved that $b_{\cal H} = (1+o(1)) \frac{n}{\ln n}$.
In other words, building a Hamiltonian cycle is possible for
Maker against essentially the same bias as building just a connected graph.

Erd\H os and Chv\'atal's winning strategy for Breaker in the
connectivity game actually isolates a vertex of Maker's graph, and
thus wins the minimum degree-$1$ game as well. Further, since a win
for Maker in the connectivity game also is a win for him in the
minimum degree-$1$ game, the results for Maker's win of the
connectivity game carry over. 
Thus, in the minimum degree-$1$ game too, the threshold bias is asymptotically equal to $\frac{n}{\ln{n}}$.
The message of this is that in positional games, having an isolated
vertex turns out to be the bottleneck for having a connected graph.
This phenomenon is familiar from the theory of random graphs,
where  Erd\H os and R\'enyi established that the sharp threshold edge number
to have a connected graph in the uniform random graph
model $\mathcal{G}(n,m)$ is the same as the one to have 
a graph with minimum degree $1$.

In fact, as already Chv\'atal and Erd\H os realized, the similarities between
random graphs and positional games are even
closer. In a positional game players are playing ``cleverly'', according to
optimal strategies and exactly one of the players has a deterministic
winning strategy, which wins against {\em any} strategy of the other player.
The situation is different if both players play ``randomly'', that is,
if both Maker and Breaker determine their moves by 
picking a uniformly random edge out of the currently free edges; then we can only talk about the ``typical'' result of the game.
The graph of this \RM will be a uniform random graph $\mathcal{G}(n,m)$ with
$m= \left\lceil \frac{{n\choose 2}}{b+1}\right\rceil$ edges. Therefore
\RM wins a particular game involving graph property ${\cal P}$
asymptotically almost surely (a.a.s.) if and only if the random graph $\mathcal{G}(n,m)$ possesses
property ${\cal P}$ a.a.s. 
Hence the classic theorem of Erd\H os and R\'enyi about the sharp connectivity threshold 
in random graphs can be reformulated in positional game theoretic terms.

\begin{thm}[Erd\H{o}s-R\'enyi, \cite{erdrenyi59}]\label{thmErdRen}
For every  $\epsilon>0$, the following holds.
\begin{enumerate}[(i)]
\item $\prob{ \mbox{\RM wins the }\left(1:(1-\epsilon)\frac{n}{\ln n}\right) \mbox{ connectivity game ${\cal C}$ } } \rightarrow 1, $
\item $\prob{ \mbox{\RB wins the }\left(1:(1+\epsilon)\frac{n}{\ln n}\right)\mbox{ connectivity game ${\cal C}$} } \rightarrow 1.$
\end{enumerate}
\end{thm}
By this theorem the threshold biases of both the random
connectivity game and the clever connectivity game are
$(1+o(1))\frac{n}{\ln n}$. This remarkable agreement means that for
most values of the bias the {\em result} of the random and the clever
game is the same a.a.s. This phenomenon is refered to as 
the {\em probabilistic intuition}.
Since similar random graph theorems also hold true for the properties
of Hamiltonicity~\cite{ajtai1985first} and having  
minimum degree $1$~\cite{erdrenyi59}, these games are also
instances where the probabilistic intuition is valid. 
One of the main directions of research in positional game theory 
constitutes of understanding what games obey the probabilistic intuition.

\subsection{Half-Random Games}\label{secHalfRandomGames}

The meaning of the probabilistic intuition is that given any
bias $b\leq (1-\epsilon) b_{\cal P}$ or $b\geq (1+\epsilon ) b_{\cal
  P}$, one could predict the winner of the ``clever'' $(1:b)$-game 
${\cal P}$ just by
running random experiments with two random players playing each
other:
whoever wins in the majority of these random games is very 
likely to have the winning strategy in the deterministic game between
the clever players.

When learning about this interpretation, it is natural to inquire
whether it is just the success of the randomized strategy in 
the clever game what is behind the whole phenomenon.
Could it be that when Maker plays uniformly at random against Breaker, who plays with 
a bias near to the threshold, then this \RM wins with high probability? 
In this paper we give precise quantitative evidence that the answer
to this question is negative. We will see that in all the games discussed
above, the random player puts himself in serious disadvantage with
playing randomly as opposed to a clever strategy. 

In what follows we investigate {\em half-random positional games}, where one of the players
plays according to the uniform random strategy against an optimal
player. There are two versions: either
Maker follows a strategy and Breaker's moves are determined randomly, or the other way around. 
We refer to the players as \CM / \RB, and \RM / \CB, respectively. 
In this paper we focus on the \RM versus \CB setup. Our approach to the other case
requires mostly different combinatorial methods and is treated in a separate paper \cite{THRG2}.

Below we define the notion of a sharp threshold bias for \RM / \CB games. For this, when we talk about a game, we
actually mean a {\em sequence of games}, parametrized with the size $n$ of the vertex set of the
underlying graph.  Similarly, when we refer to a strategy of \CB, we mean a
{\em sequence of strategies}.

It will turn out that when Maker plays randomly, the disadvantage of making random moves
outweighs his huge advantage  inherent in the $(1:1)$ games, and the half-random
bias  needs to tilt in his favor. This motivates the following
definition.

\begin{defin}\label{def:sharpThreshold}
We say a function $k:\natsz\mapsto\natsz$ is a
\textit{sharp threshold bias} of the $(a:1)$ half-random positional game between \RM and 
\CB, if for every $\epsilon > 0$ it satisfies the following two conditions
\begin{enumerate}[(a)]
\item \RM wins the $((1+\epsilon) k(n):1)$-biased game a.a.s. 
against any strategy of \CB, and 
\item \CB has a strategy against which \RM loses the \linebreak[4]$\left(\left(1-\epsilon\right)k(n):1\right)$-biased 
game a.a.s.
\end{enumerate}
\end{defin}

{\bf Remarks.} 
{\bf 1.} 
Our paper is mostly about the failure of the uniformly random strategy
against a clever player in various classical graph games. There are
other natural games where the situation is completely different and 
the uniformly random strategy is close to being optimal.
Bednarska and \L uczak \cite{bednarska2000biased} consider the $H$-building game 
${\cal K}_H$, where Maker's goal is to occupy a copy of a fixed graph $H$.  
Even though their paper is about the classical game scenario with clever players,
their results also imply that the half-random 
$(1:b)$ $H$-game ${\cal K}_H$ between \RM and \CB
has a threshold bias around $n^{\frac{1}{m_2(H)}}$, where 
$m_2(H) = \max_{K\subseteq H, v(K)\geq 3}\frac{e(K)-1}{v(K)-2}$. 
Bednarska and \L uczak not only prove that \RM succeeds against 
a bias $cn^{\frac{1}{m_2(H)}}$ for some small constant $c$ a.a.s., but also that 
even a \CM would not be able to do much better.
That is, they give a strategy for \CB  to prevent
the creation of $H$ by \CM with a bias $Cn^{\frac{1}{m_2(H)}}$, 
where $C$ is some large constant. 
The $H$-game is an instance of a game where the threshold
bias for the clever game is of the same order of magnitude as for the half-random game
--- very much unlike the games we consider in this paper.

{\bf 2.} Half-random versions of other positional games were
also considered earlier in different context. The well-studied notion
of an Achlioptas process can be cast as the {\tt RandomWaiter}-{\tt CleverClient} version of 
the classic Picker-Chooser games introduced by
Beck~\cite{BeckCombinatorica-2002} (and renamed recently to
Waiter-Client by Bednarska-Bzdega, Hefetz, \L uczak~\cite{BHL}). In a $(1:1)$
Waiter-Client game the player Waiter chooses two, so far unchosen edges of $K_n$
and offers them to the player called Client, who selects one of them
into his graph. Waiter wins when Client's graph has property ${\cal P}$.
A substantial amount of work~\cite{BohmanFrieze, RiordanWarnke} was
focused on determining 
{\em how long} does it take for {\tt RandomWaiter} to win when the
property ${\cal P}$ is to have a connected component of linear size. 
Bohman and Frieze~\cite{BohmanFrieze} gave a simple strategy for 
{\tt CleverClient} to significantly delay the win of {\tt
  RandomWaiter} 
compared to the well-known threshold from the work of Erd\H os and R\'enyi 
in the game where both players play randomly.

\subsection{Results}
We first show that if $a\leq(1-\epsilon) \ln{\ln{n}}$, then a simple and natural strategy of \CB 
allows him to isolate a vertex in \RM's graph a.a.s., and therefore win the degree-$1$ game. 
Then we establish that this threshold is asymptotically tight for all the games 
we are considering in this paper.
\begin{thm}\label{thm:RMCB-mindegree-k} Let $k$ be a positive integer.
The sharp threshold bias for the $(a:1)$  minimum degree-$k$
game between \RM and \CB is $\ln\ln n$. 
\end{thm}

\begin{thm}\label{thm:RMCB-connectivity}
The sharp threshold bias for the $(a:1)$  connectivity game between \RM and \CB is $\ln\ln n$. 
\end{thm}

\begin{thm}\label{thm:RMCB-hamiltonicity}
The sharp threshold bias for the $(a:1)$  Hamiltonicity game between \RM and \CB is $\ln\ln n$. 
\end{thm}

On the one hand these theorems show that mindless random strategies
are very ineffective for the games we consider here, where the goal is ``global''. 
As discussed earlier, randomized strategies are shown to be close to optimal  for games where the 
goal of Maker is ``local'', for example when the goal of Maker is 
to build a fixed subgraph $H$ \cite{bednarska2000biased}. 
On the other hand, these theorems establish that the bottleneck for winning
connectivity and Hamiltonicity in half-random games is to be able to
win the minimum degree-$1$ game. This is similar to the phenomenon
that occurs in the fully random and the fully clever scenario.\\

{\bf Remarks.} 
The results of this paper and of \cite{THRG2} are based on the Master thesis
of the first author~\cite{Jonas-thesis}. Recently, Krivelevich and
Kronenberg~\cite{Krivelevich-Kronenberg} also studied half-random
games independently (both in the \CM-\RB and the \RM-\CB setup). 
For the \RM-\CB setup they determine the order of magnitude of 
the half-random threshold bias of the Hamiltonicity and the  $k$-connectivity
game. Here we manage to pin down the constant factor for Hamiltonicity
and the minimum degree-$k$ games. In the conclusion section we also indicate how
the similar sharp threshold result for the $k$-connectivity game can
be obtained easily from our proof technique. \\
In~\cite{THRG2} we determine the sharp threshold bias of the perfect
matching and the Hamiltonicity games in the \CM-\RB setup. 
Krivelevich and Kronenberg~\cite{Krivelevich-Kronenberg} obtain
analogous results with different methods.
%

\subsection{Terminology and organization}

We will use the following terminology and conventions.
A \textit{move} consists of claiming one edge. \textit{Turns} are taken alternately, 
one turn can have multiple moves. For example: With an $(a:1)$ bias,
Maker has $a$ moves per turn, while Breaker has $1$ move. 
A {\em round} consists of a turn by Maker followed by a turn by Breaker. 
By a strategy we mean a set of rules which specifies what the
player does in any possible game scenario. For technical reasons 
we {\em always} consider strategies that last until there are no free
edges. This will be so even if the player has already won, already
lost, or his strategy description includes ``then he forfeits''; in
these cases the strategy just always occupies an arbitrary free edge,
say with the smallest index. 
The {\em play-sequence} $\Gamma$ of length $i$ of an actual game between Maker and Breaker 
is the list $(\Gamma_1, \ldots , \Gamma_i) \in E(K_n)^i$ of the first $i$ edges that were occupied during the game by either of the players, 
in the order they were occupied. 
We make here the convention that a player with a bias $b> 1$ occupies 
his $b$ edges within one turn in succession and these are noted in the play-sequence 
in this order
(even though in the actual game it makes no difference in what order one player's moves are
occupied within one of his turns). 
We denote Maker's graph after $t$ rounds with $G_{M,t}$ and similarly Breaker's graph with $G_{B,t}$. Note that these graphs have $at$ and $bt$ edges respectively. 
We will use the convention that Maker goes first. This is more of a notational 
convenience, since the proofs can be easily adjusted to Breaker going first, and yielding 
the same asymptotic results. 
We will routinely omit rounding signs, whenever they are not crucial in affecting our asymptotic statements.

We introduce the useful notion of the \textit{permutation strategy} in the next section, and prove Theorems \ref{thm:RMCB-mindegree-k},  \ref{thm:RMCB-connectivity} and \ref{thm:RMCB-hamiltonicity} in Section \ref{sec:RMCB}.

\section{The permutation strategy}\label{sec:permutation}
In this section, we introduce an alternative way to think of half-random games which 
will be important in many of our proofs. 
One feature that makes half-random games more difficult to study 
than random games is that the graph of the random player is {\em not} 
uniformly random: the moves of the clever player affect it. 
Our goal is still to be able to somehow compare it to the uniform random graph 
$G(n,m)$ with the appropriate number of edges and draw conclusions from the rich
theory of random graphs.

Any of the players in a positional game can use a permutation $\sigma \in S_{E(K_n)}$, i.e. $\sigma: \left[ {n\choose 2} \right] \rightarrow E(K_n)$,
of the edges of $K_n$ 
for his strategy as follows. The player following the {\em permutation strategy} $\sigma$ 
is scanning through the list $(\sigma(1), \ldots , \sigma ({n\choose 2}))$ during the game and in each of his moves he occupies the next free edge on it (that is, the next edge
which was not yet occupied by his opponent). 
The permutation strategy gives rise to a natural randomized strategy for 
\RM when he selects the permutation uniformly at random. 
It turns out that playing according to this random permutation strategy is equivalent 
to playing according to the original definition of \RM's strategy 
(i.e., always choosing uniformly at random from the remaining free edges). 
 
The following proposition formalizes this. Intuitively it is quite
clear, in \cite{THRG2} we give a formal proof of a more general statement.
Here we only state the special case we need. 
Since the goal of the game is not relevant here, we state the proposition for graph games in general.
\begin{prop}\label{prop:permutationStrategy}
For every strategy $S$ of \CB in a $(a:b)$-game on $E(K_n)$ the following is true.
For every 
$m\leq {n\choose 2}$
and every sequence 
$\Gamma = (\Gamma_1, \ldots , \Gamma_m)$ of distinct edges,
the probability that 
$\Gamma$ is the play-sequence of a half random game between 
\CB  playing according to strategy $S$ and \RM
is equal to the probability that $\Gamma$ is the play-sequence of the game when 
$\RM$ plays instead according to the random permutation strategy.
\end{prop}

For $1\leq m\leq {n\choose 2}$ and a permutation $\sigma \in
S_{E(K_n)}$, let $\Gsigma{\sigma}{m}\subseteq K_n$ be the subgraph with edge set $E(\Gsigma{\sigma}{m}) = \{ \sigma(i) : 1\leq i \leq m\}$.
Note that if $\sigma$ is a permutation chosen uniformly at random out
of all permutations, then $\Gsigma{\sigma}{m}$ is distributed like the random
graph $G(n,m)$.
If \RM plays a particular game according to a permutation
$\sigma \in S_{E(K_n)}$ and the last edge he takes
in round $i$ is $\sigma (m_i)$, then \RM's graph after round $i$ is contained in $\Gsigma{\sigma}{m_i}$.  
Here $m_i \geq ia$, but the actual value of it depends on the strategy
of \CB and the 
permutation $\sigma$ itself. Since \CB occupied
$ib$ edges so far and these are 
the only edges \RM possibly skips from his permutation, we also have that $m_i \leq i(a+b)$.
Hence  \RM's graph after the $i$th round is always contained in the random
graph $\Gsigma{\sigma}{i(a+b)}$.

\section{\CB vs \RM}\label{sec:RMCB}

In this section, we prove Theorems  \ref{thm:RMCB-mindegree-k},
\ref{thm:RMCB-connectivity}, and \ref{thm:RMCB-hamiltonicity}. 
We start with showing that a.a.s. \CB is able to isolate a vertex in
\RM's graph if the bias of \RM is not too large. This provides 
the lower bound on the sharp thresholds in all the games we study 
and is the topic of the next subsection. We treat the 
upper bounds in Subsections~\ref{sec:RM-structure} and \ref{sec:RMHamiltonicity}.

\subsection{\CB isolates a vertex of \RM}\label{sec:CBWin}
In this subsection we prove the following theorem.
\begin{thm}\label{thm:CBWinDegree1}
Let $\epsilon>0$ and $a\leq (1-\epsilon)\ln{\ln{n}}$. Then there exist a strategy for \CB, such that he a.a.s. wins the $(a:1)$-biased minimum degree-$1$ game against \RM.
\end{thm}

\begin{proof}
Let $v_1, v_2, \ldots , v_n$ be the vertices of the underlying complete graph. 
\CB's strategy is rather simple. \CB identifies the vertex 
$v_i$ of smallest index which has degree $0$ in Maker's graph.
Then he occupies the free edges incident to $v_i$, one by one, in an increasing order of 
the indices of their other endpoint. (We refer to this process as {\em \CB trying to isolate 
$v_i$.})  If he succeeds in occupying all $n-1$ edges incident to $v_i$, then he won 
the game. Otherwise, that is if \RM occupied an edge incident to $v_i$ while \CB was trying to 
isolate it, \CB iterates: he identifies a new vertex he tries to
isolate. 
In this case we say that {\em \CB failed to isolate $v_i$}. 
If \CB fails to isolate
\[k(n)=k:=(1-\epsilon)\frac{\ln{n}}{4\ln{\ln{n}}}\]  
vertices then he forfeits.

Recall the permutation strategy for the random player of Section \ref{sec:permutation}, based on a random permutation of the edges of $E(K_n)$. Let us denote by ${\cal W}$ the set of those permutations for \RM which 
would result in a win for \CB using this described strategy. 
Note that for $k$ tries, Breaker spends at most $(n-1)k<nk$ edges (and therefore turns)
and hence the presence of a permutation $\sigma$ in ${\cal W}$ is determined by its first 
$(a+1)nk$ edges. 

Let ${\cal A}$ denote the set of those permutations $\sigma$ 
for which the graph $\Gsigma{\sigma}{(a+1)nk}$ of the first $(a+1)nk$ edges 
has an isolated vertex. Since
 $$(a+1)nk \leq (\ln{\ln {n}}+1)n\frac{(1-\epsilon)\ln{n}}{4\ln{\ln{n}}}
\leq (1-\epsilon)\frac{1}{2}n\ln{n},$$
the classic result of Erd\H os and R\'enyi~\cite{erdrenyi59} on
the sharp threshold in $G(n,m)$ for the minimum degree being at least $1$ implies
the following.
\begin{lem}\label{lem1isol} ${\cal A}$ occurs a.a.s. 
\end{lem}

The following lemma guarantees that, conditioned on ${\cal A}$, \CB 
tries to isolate $k$ vertices or wins already earlier.
\begin{lem}\label{lemkTries} For every $\sigma \in {\cal A} \setminus {\cal W}$ 
\CB tries to isolate $k$ vertices.
\end{lem}
\begin{proof} 
For any permutation $\sigma \in {\cal A}$, the graph
$\Gsigma{\sigma}{(a+1)nk}$ contains the graph of \RM up to the point when \CB tries and fails to isolate
at most $k$ vertices. On the other hand $\Gsigma{\sigma}{(a+1)nk}$ does have an isolated
vertex by the definition of ${\cal A}$, 
so \CB did not run out of isolated vertices by the time he failed to
isolate his $(k-1)$th vertex. 
\end{proof}

The main ingredient of our proof is an estimation of the probability that
\CB fails to isolate his $j$th vertex, given that he already failed to isolate 
the first $j-1$ vertices.
Let ${\cal D}_0:=S_{E(K_n)}$ be the set of all permutations, and for $1\leq j\leq k$, 
let ${\cal D}_j$ denote the event (set of permutations) that induces a game where 
\CB tries and fails to isolate at least the first $j$ vertices. 
Notice that ${\cal D}_0 \supseteq {\cal D}_1 \supseteq \cdots \supseteq {\cal D}_k$. 
Our eventual goal is to show that ${\cal D}_k\cap {\cal A}$ is very small.
To achieve this we bound $|{\cal D}_j \cap {\cal A}|$ in terms of $|{\cal D}_{j-1} \cap {\cal A}|$.

\begin{prop}\label{propDjProb} For every $n$ large enough and every $j, 1\leq j \leq k$, we have
$$|{\cal D}_j \cap {\cal A}| \leq \left(1-\frac{1}{\ln^{1-\epsilon^2/2} n}\right) 
|{\cal D}_{j-1} \cap {\cal A}|.$$	
\end{prop}
Before we prove the proposition, let us show how it implies our theorem.
Following the strategy defined above, \CB forfeits either if he fails
every one of his first $k$ tries to isolate a vertex, or if there is
no more vertex of Maker-degree $0$. 
We saw in Lemma \ref{lemkTries}, that for any permutation in 
${\cal  A}$ the latter one is not an option: \CB has to fail at least $k$ times before he runs out of 
vertices he can try. Therefore, using Proposition \ref{propDjProb}, we obtain
\begin{eqnarray*}
\frac{|{\cal D}_k\cap {\cal A}|}{|{\cal A}|} &\leq & \left(1- \frac{1}{(\ln n)^{1-\epsilon^2/2}}\right) \frac{|{\cal D}_{k-1}\cap {\cal A}|}{|{\cal A}|}\\  
&\leq & \left(1- \frac{1}{(\ln n)^{1-\epsilon^2/2}}\right)^k\\
&\leq& e^{-k(\ln n)^{-(1-\epsilon^2/2)}}\\
&\leq & e^{-(1-\epsilon)\frac{\left(\ln{n}\right)^{\epsilon^2/2}}{4\ln{\ln{n}}}}
\rightarrow 0.
\end{eqnarray*}
Finally, since ${\cal A}$ holds a.a.s.\ by Lemma \ref{lem1isol}, we also have
\[\prob{\text{\CB wins}}\geq \prob{\overline{{\cal D}_k}\given {\cal
  A}}\prob{{\cal A}} \ \rightarrow \ 1.\]
\end{proof}

To complete the proof of Theorem~\ref{thm:CBWinDegree1} we need to
prove Proposition~\ref{propDjProb}.
For that there is a subtle technicality that we have to take care of. If we assume that ${\cal A}$ holds, 
we use knowledge of the first $(a+1)nk$ random edges of our permutation
and thus knowledge of \RM's moves up until the turn $nk$. Therefore, if we consider the distribution of the next move of \RM among the free 
edges {\em before} turn $nk$, conditioned under ${\cal A}$, this distribution might not be uniform anymore. 
 For example, if there is only one vertex $\tilde{v}$ left with degree $0$ in \RM's graph, 
 then the probability that \RM chooses an edge incident to $\tilde{v}$, under the condition that ${\cal A}$ holds, is $0$. 
 However, while some edges may have very low probability to be chosen by \RM, we can 
 show that there are no edges that have a particularly high probability to be picked.

For a starting edge sequence $\pi \in S_{E(K_n)}^{(m)}$ of length $m$,
let ${\cal A}(\pi) \subseteq {\cal A}$ denote the set of permutations 
$\sigma \in {\cal A}$ with  initial segment equal to $\pi$. 
Given an edge sequence $\eta \in S_{E(K_n)}$ and a strategy $S$ of \CB,  
we say an edge $e\in E(K_n)$ to be {\em $(S,\eta)$-Maker}  if it
is taken by \RM when he plays according to $\eta$ against strategy $S$. 
Let ${\cal A}(\pi ; S;e) \subseteq {\cal A}(\pi)$ denote the set of permutations $\eta \in S_{E(K_n)}$
which start with $\pi$ and after that the next $(S,\eta)$-Maker edge is $e$.

\begin{lem}\label{technicalLemma} For every $\epsilon >0$ the following holds for large enough $n$.
For every strategy $S$ of \CB,  positive integer $m\leq (a+1)nk$, 
starting permutation $\pi \in S_{E(K_n)}^{(m)}$ of length $m$ and 
edge $e\in E(K_n)$ we have that 
$$|{\cal A}(\pi; S; e)| \leq (1+\epsilon) \frac{2}{n^2} |{\cal A} (\pi)|.$$
\end{lem}
\begin{proof}
We can assume that $e$ is still  unoccupied after the permutation strategy
has been played according to $\pi$, otherwise
the statement is trivial (since the set ${\cal A}(\pi; S; e)$ is empty).

We partition the sets ${\cal A}(\pi)$ and ${\cal A}(\pi; S;e)$
according to the sequence of edges that come after $\pi$ in the
permutations until the first $(S,\eta)$-Maker edge. 
Let $\pi'$ be an arbitrary extension of $\pi$ with a sequence $\tau$ containing only such edges 
which were occupied by \CB when the permutation strategy was played according to $\pi$. 
($\pi'=\pi$ is also possible.)
Note that the length of $\pi'$ is at most $m +  \frac{m}{a} \leq 2(a+1)nk = o(n^2)$.

Let ${\cal \hat{A}(\pi';S)} \subseteq {\cal A}(\pi')$ be the set of those permutations $\eta$ which 
start with $\pi'$ and continue with an $(S, \eta)$-Maker edge. 
Let $\hat{{\cal A}}_e (\pi',S)\subseteq \hat{{\cal A}}(\pi',S)$ be the set of permutations where the edge $e$ comes immediately after $\pi'$.
Unless otherwise stated, from now on we consider $\pi'$ fixed and suppress it in the arguments of 
$\hat{{\cal A}}_e$ and $\hat{{\cal A}}$. 

To show the upper bound of the Lemma, we will find for any permutation $\eta\in\mathcal{\hat{A}}_e$ 
many different permutations in $\mathcal{\hat{A}}$. 
For any such  $\eta$ and edge $f\in E(K_n)$ 
we denote by $\eta_e^f $ the edge permutation with the positions of $e$ and $f$ interchanged. 
Let ${\cal M}(\eta)$ be the set of those permutations $\eta_e^f$ which are in $\mathcal{\hat{A}}$. That is,
\[{\cal M}(\eta):=\left\{\eta_e^f: f\in E(K_n), \eta_e^f \in\mathcal{\hat{A}}\right\}.\] 
There are three possible reasons why a permutation $\eta_e^f$ would not be in ${\cal M}(\eta)$:
\begin{enumerate}
\item Any permutation in $\mathcal{\hat{A}}$ must start with $\pi'$, hence we are not allowed to swap $e$  with any edge that comes up in $\pi'$.
The number of these forbidden edges is $m \leq (a+1)nk = o(n^2)$.
\item In any permutation $\eta\in\mathcal{\hat{A}}$ the edge following $\pi'$ must be $(S,\eta)$-Maker, hence we cannot swap $e$ with 
any edge claimed by \CB up to this point. There are at most
$\frac{m}{a} = o(n^2)$ such edges. 
\item Finally, the graph formed by the first $(a+1)nk$ edges of any edge permutation in $\mathcal{\hat{A}}$ must have an isolated vertex. So if $\Gsigma{\eta}{(a+1)nk}$ 
had only one isolated vertex $\tilde{v}$, we might not be able to swap $e$ with an edge $f$ incident to $\tilde{v}$, since then $\Gsigma{\eta_e^f}{(a+1)nk}$ 
might not have an isolated vertex anymore. So we forbid a swap with the $n-1=o(n^2)$ incident edges to the last isolated vertex of $\Gsigma{\eta} {(a+1)nk}$.
\end{enumerate}
Swapping $e$ with any edge that is {\em not} in these three categories leads to an edge permutation in $\hat{\cal A}$. Therefore,
$\left|{\cal M}(\eta)\right|\geq {n\choose 2}-o(n^2).$
By definition ${\cal M}(\eta) \subseteq \hat{{\cal A}}$ for every permutation $\eta\in\mathcal{\hat{A}}_e$. 
The sets ${\cal M}(\eta) $ and ${\cal M}(\zeta)$ are disjoint for $\eta \neq \zeta$,
as clearly $\eta_e^f = \zeta_e^{f'}$ is only possible if $f=f'$ and $\eta=\zeta$.
Hence for the cardinalities we have
\begin{equation}\label{eq:tech-proportion}
\left|\mathcal{\hat{A}}\right|\geq \sum_{\eta\in\mathcal{\hat{A}}_e}\left|{\cal M}(\eta)\right|\geq\left|\mathcal{\hat{A}}_e\right|\left({n\choose 2}-o(n^2)\right).
\end{equation}
Now recall that $\hat{{\cal A}} = \hat{{\cal A}}(\pi';S)$ and  $\hat{{\cal A}}_e = \hat{{\cal A}}_e(\pi',S)$ where $\pi'$ was an arbitrary, but fixed 
extension of $\pi$ with an edge sequence $\tau$ containing only edges \CB took up to playing according to $\pi$.

Our focus of interest, the sets ${\cal A} (\pi)$ and ${\cal A} ( \pi; S;e)$ are disjoint unions of the sets $\hat{{\cal A}}(\pi',S)$
and $\hat{{\cal A}}_e(\pi',S)$, respectively, where the disjoint union is taken over all extensions $\pi'$ of $\pi$ with distinct 
edges
which were occupied by \CB in the game played according to $\pi$.
Therefore Equation~\eqref{eq:tech-proportion} is also valid for them and hence,
\[|{\cal A} (\pi;S;e)| \leq\frac{1}{{n\choose 2}-o(n^2)}|{\cal A} ( \pi)|\leq\frac{2(1+\epsilon)}{n^2} |{\cal A} ( \pi)|\]
for $n$ large enough, which is the statement of the lemma.
\end{proof}

With Lemma \ref{technicalLemma} proven, we can return to the main line of reasoning.
\begin{proof}[Proof of Proposition \ref{propDjProb}]
Let $\sigma \in {\cal A} \cap {\cal D}_{j-1}$ and let $w_1, \ldots,
w_{j-1}, w_j$ be the vertices \CB tries to isolate, in this
order, when he plays against $\sigma$. Note that the first $j-1$ of these vertices do exist 
because $\sigma \in {\cal D}_{j-1}$ and then $w_j$ also exists by Lemma~\ref{lemkTries}. 
We define $\pi = \pi(\sigma)$ to be the initial segment of $\sigma$, which ends with the 
last edge \RM takes in the round where he occupies his first edge incident to $w_{j-1}$. 
The length of $\pi$ is at most $(a+1)(n-1)(j-1)$, so we will be able
to use Lemma~\ref{technicalLemma}. Let $\Pi$ be the set of all such $\pi$, i.e.
$$\Pi=\left\{\pi(\sigma):\sigma\in {\cal A} \cap {\cal D}_{j-1}\right\}$$

We classify the permutations $\sigma \in {\cal D}_{j-1} \cap {\cal A}$ according to these
initial segments. 
We will prove that for all $\pi\in\Pi$,
$$|({\cal D}_{j}\cap {\cal A}) (\pi)| \leq \left(1-\frac{1}{\ln^{1-\epsilon^2/2} n}\right) |({\cal D}_{j-1}\cap {\cal A}) (\pi)|$$
and the statement follows since $({\cal D}_{j}\cap {\cal A})$ is the disjoint union of the
$({\cal D}_{j}\cap {\cal A}) (\pi)$ when the disjoint union is taken
over all $\pi\in\Pi$. The union is disjoint since no element of $\Pi$
is the prefix of another.

Let us fix an arbitrary initial segment $\pi\in\Pi$. After $\pi$ has
been played out, \CB immediately identifies the next 
vertex $w_j$ he will try to isolate. Suppose that $r\leq n-1$ edges incident to $w_j$ are free,
that is \CB occupied already $n-1-r$ edges incident to $w_j$ during his previous tries to isolate a vertex, while \RM occupied
none.
In the next round \CB occupies the free edge from $w_j$ to the vertex with the smallest
index. Then \RM has $a$ random edges and the question is whether he hits 
any of the remaining $r-1$ free edges incident to $w_j$.

Note that all permutations starting with $\pi$ are in ${\cal D}_{j-1}$, and thus \linebreak$\left({\cal D}_{j-1}\cap {\cal A}\right)(\pi)={\cal A}(\pi)$. Therefore, the number of such permutations where \RM in his next move hits one of these 
edges is at most
\[(r -1) \frac{2(1+\epsilon)}{n^2} |({\cal D}_{j-1}\cap {\cal A}) (\pi)| \]
by Lemma~\ref{technicalLemma}.
Then the number of permutations where \RM did not play any of these edges is  at least
\[\left(1-(r-1)\frac{2(1+\epsilon)}{n^2}\right)|({\cal D}_{j-1}\cap {\cal A}) (\pi)| .\]
We repeat the process for the $a$ moves of \RM, always taking a new set $\Pi$, letting the initial segment $\pi$ run until \RM's last move each time,
always conditioning that \RM has not yet claimed an edge incident to $w_j$ (i.e. allowing only such $\sigma$).
Applying Lemma~\ref{technicalLemma} iteratively, the number of  permutations where none of \RM's $a$ edges are incident to $v_j$ is at least 
\[\left(1-(r-1)\frac{2(1+\epsilon)}{n^2}\right)^a|({\cal D}_{j-1}\cap {\cal A}) (\pi)| .\]

In order to estimate the
number of permutations in which \RM does not take any edges incident to $w_j$ and hence
\CB isolates $w_j$, we repeat the above process over the relevant $r -1$ turns. The calculation is identical for each turn, except that the number of vacant edges incident to $w_j$ decreases. Taking the product over these $r-1$ turns, we obtain
\begin{eqnarray*}
|(\overline{{\cal D}_j} \cap {\cal A})(\pi)| &\geq& \prod_{\ell=1}^{r-1}\left(1-\frac{2(1+\epsilon)\ell}{n^2}\right)^a|({\cal D}_{j-1} \cap {\cal A})(\pi)|\\
&\geq & e^{-a\left(\sum_{\ell=1}^{r-1}\frac{2(1+\epsilon)\ell }{n^2}\right)-
a\left(\sum_{\ell=1}^{r-1}(\frac{2(1+\epsilon)\ell}{n^2})^2\right)} 
({\cal D}_{j-1} \cap {\cal A})(\pi)|\\
& \geq &
e^{-a(1+\epsilon) - O(\frac{a}{n}) }({\cal D}_{j-1} \cap {\cal A})(\pi)|\\
& \geq &
e^{-(1-\epsilon^2)\ln \ln n - O(\frac{\ln\ln n}{n}) }({\cal D}_{j-1} \cap {\cal A})(\pi)|\\
&\geq& \left(\ln^{-(1-\epsilon^2/2)}{n}\right)\left|({\cal D}_{j-1} \cap {\cal A})(\pi)\right|
\end{eqnarray*}
using $r\leq n-1$.

\end{proof}

\subsection{\RM builds a connected graph with minimum degree at least $k$} \label{sec:RM-structure}

The proofs of the upper bound for all games in Theorems~\ref{thm:RMCB-mindegree-k}, \ref{thm:RMCB-connectivity} and \ref{thm:RMCB-hamiltonicity} all start out the same.
First we establish that the vertices with many incident edges occupied by \CB (called ``bad vertices'') 
are well-connected to the rest of the graph (called ``good vertices''). 
Then we go on to show that the graph of Maker on the good vertices is very close to being a uniformly 
random graph. Then the upper bound in the min-degree-$k$ and connectivity game follows easily.
To prove the existence of a Hamilton cycle is somewhat more technical and is presented in its separate section.

Let $\epsilon >0 $ fixed and $a\geq(1+\epsilon)\ln{\ln{n}}$.
We consider the $(a:1)$-game on $K_n$ between \RM and \CB playing according to an arbitrary but fixed 
strategy $S$.  
First we introduce some notation. Let us fix a parameter, 
$\alpha$, $0<\alpha<\frac{1}{2}$ sufficiently small, 
such that  
\begin{equation} \label{eq:alpha}
(1+\epsilon)(1-3\alpha)^2 > 1 +\frac{\epsilon}{2}.
\end{equation}
We will consider the first  $t:=\frac{\alpha}{2}n\ln{n}$ rounds of the game and show that \RM finishes 
his job within his first $t$ turns, a.a.s.\\
The key idea of the proof is to divide the vertices in categories, based on how many incident edges \CB claims.
We call a vertex \textit{$\alpha$-bad}, if  its degree in \CB's graph is $3\alpha n$ or more and otherwise 
we call it  \textit{$\alpha$-good}. Since throughout this section $\alpha$ is fixed, we suppress it and talk about bad and good vertices.

An important observation is that during the first $t$ rounds the total degree in \CB's graph does not exceed $2t$ 
hence there cannot be more than \[\frac{2t}{3\alpha n} = \frac{2}{3\alpha n}\cdot \frac{\alpha n\ln{n}}{2} \leq \ln{n}\]
bad vertices. In other words, the vast majority of vertices, namely $n-\ln{n}$, are still good after $t$ turns.

\subsubsection{Connecting the bad vertices}

To grade the transition of a good vertex into a bad one we define the concept of a candidate vertex.
We say that a vertex $u$ is
\begin{enumerate}[(i)]
\item an {\em early candidate} if \CB claimed his $\alpha n$-th edge incident to $u$ before round $t-(1-\alpha)n$, and 
\item a {\em late candidate}, if $u$ is not an early candidate and \CB claims his $2\alpha n$-th edge incident to $u$ in a turn $s$ 
with $t-(1-\alpha)n\leq s\leq t-\alpha n$.
\end{enumerate}

Observe that every vertex that is bad at turn $t$ had to become (early or late) 
candidate in a turn $s\leq t-\alpha n$.
Indeed, if a vertex $u$  is bad then it must have had degree at least $2\alpha n$ in \CB's graph at round 
$t-\alpha n$. If $u$ is not an early candidate then it got its $\alpha n$-th edge and hence also its $2\alpha n$-th edge
after round $t - (1-\alpha)n$, so it is a late candidate. 

Note that only good vertices can become candidates and 
once a vertex becomes candidate it stays that way till the end. 
This also means in particular that {\em every} bad vertex is {\em also} a candidate.

Let us now fix an integer $k\geq 1$. In most definitions and statements that follow 
$k$ appears as a parameter, but we will suppress it if this creates no confusion.
Let us define an auxiliary digraph $D_k=D$ which is built throughout the first $t$ rounds of the game on the vertex set
$[n]$ of the $K_n$ the game is played on.  For this, we imagine that \RM occupies the $a$ edges within 
each of his turn one after another, so we can talk, without ambiguity, about an edge being occupied 
{\em before} another. 
The digraph $D$ has no edges at the beginning of the game. During the game we add edges to $D$ 
in the following two scenarios:
\begin{enumerate}[(1)] 
\item whenever a good vertex $u$ becomes (early or late) candidate at some turn of \CB, we immediately 
add to $D$ up to $k$ arbitrary arcs $(u,v)$,
such that $uv$ is occupied by \RM already, the vertex $v$ is not a candidate, and 
$d^-_{D}(v)=0$. If there are less than $k$ such edges incident to $u$ we add them all.
\item whenever \RM occupies an edge $uv$ in the game where the vertex $u$ is a candidate, 
we add the arc $(u,v)$ to $D$ if the vertex $v$ is not a candidate, $d^+_{D}(u) <k$,   and $d^-_{D}(v)=0$.
\end{enumerate}
We call the edges of $D$ {\em saviour edges}. At any point of the game an arc $(u,v)$ is called 
a \textit{potential savior edge for $u$} if the edge $uv$ is unoccupied in the game, the vertex $u$ turned candidate 
already, the vertex $v$ did not, $d^+_{D}(u) <k$,   and $d^-_{D}(v)=0$.

\begin{lem}\label{lemDAcyclic}
For the maximum in- and out-degree of $D$ we have $\Delta^-(D)\leq 1$ and $\Delta^+(D) \leq k$.
The underlying graph of $D$ is an acyclic subgraph of Maker's graph. 
\end{lem}
\begin{proof} The bounds on the maximum in- and out-degree immediately follow from the rules in
(1) and (2), as it does that the underlying graph is a subgraph of Maker's graph. 
For acyclicity it is enough to check that at the time an arc $(u,v)$ is added to $D$, the tail vertex $v$ 
is isolated in $D$. 
For this first note that $d^-_{D_k}(v)=0$, so $v$ has no incoming arc. For $v$ to have some out-going 
arc $(v,w)$ in $D$, $v$ has to be a candidate already. But adding the arc $(u,v)$ requires that $v$ did not turn candidate yet, a contradiction.
Hence $v$ was isolated before the addition of $(u,v)$.
\end{proof}

In the remainder of this section we show that a.a.s. every bad vertex has out-degree $k$ in $D$.  

\begin{lem}\label{LemKOutgoingFast}
For
every vertex $u\in [n]$, the following holds:
\begin{enumerate}[(i)]
\item\label{lemKOutgoingFastEarly} If $u$ is an early candidate,
then $(1-3\alpha)n$ rounds after it turned candidate, $d^+_{D}(u) = k$  with probability $1-o(\ln^{-1} n)$.
\item\label{lemKOutgoingFastLate} If $u$ is a late candidate, 
then $\alpha n$ rounds after it turned candidate, $d^+_{D}(u) = k$ with probability $1-o(1)$.
\end{enumerate}
\end{lem}
\begin{proof}
Let $u$ be a vertex which turns candidate in round $t_u$ and assume that  $d^+_{D}(u) < k$ at that point 
(otherwise we are done). Let $\ell:= k-d^+_{D}(u)$ be the number of saviour edges $u$ must still collect.
For $i = 1, \dots, (1-3\alpha)n$ let $E_i$ be  the event that no saviour edge from  $u$ is added to $D$ at turn 
$t_u+i$ of \RM. If at least $\ell$ of the events $E_i$ do not hold, 
then $u$ has out-degree $k$.\\
We are interested in the probability
$$ p_i:=\prob{ E_i\given \text{there are less than $\ell$ rounds $t_u+j$, $j<i$, s.t. $E_j$ holds} }.$$
How many potential saviour out-edges are there for $u$? 
Since $u$ turns candidate in round $t_u$ and \CB claims at most one incident edge per turn, by round $t_u+i$ \CB has 
claimed at most $2\alpha n+i$ edges incident to $u$ (this holds for both early and late candidates). 
There are at most $2t/\alpha n = \ln n$ vertices that turned candidate before round $t$ 
and each of these might have at most $k$ outneigbors. These are at most $(k+1)\ln n$ further vertices 
to where there is no potential saviour edge from $u$.
Thus, there are at least $n-1- (k-1) - (2\alpha n + i) -(k+1)\ln n \geq (1-3\alpha)n-i$ potential $k$-savior edges from 
$u$. Hence every edge \RM claims in round $t_u+i$ has probability at least $\frac{(1-3\alpha)n-i}{{n\choose 2}}$ to be a 
savior edge for $u$. This implies that $p_i$, the probability that \RM does not claim any savior edge in his $a$ moves 
this round can be estimated as
\begin{align*}
p_i&\leq\left(1-\frac{(1-3\alpha)n-i}{{n\choose 2}}\right)^a < e^{-\frac{2a}{n}\left(1-3\alpha - \frac{i}{n}\right)}
\end{align*}
Conversely, there are at least ${n\choose 2}-(a+1)t= {n\choose 2} - o(n^2)$ free edges in total, therefore 
\begin{align*}
p_i&\geq\left(1-\frac{n-1}{{n\choose 2}-o(n^2)}\right)^a=\left(1-\frac{2+o(1)}{n}\right)^a > e^{-\frac{2a}{n}(1+o(1))} 
\end{align*}
We now consider the first $Cn$ rounds after $t_u$, for a constant $0<C<1$ (for (\ref{lemKOutgoingFastEarly}) 
we choose $C=1-3\alpha$, for (\ref{lemKOutgoingFastLate}) we choose $C=\alpha$).
Fix now an integer $j$, $0\leq j \leq \ell -1$ and let $q_j$ be the probability that there are exactly $j$ rounds where 
a saviour edge from $u$ is occupied by \RM.  Then the probability that $d^+_D(u) < k$ after round $t+Cn$ is
$\sum_{j=0}^{\ell-1} q_j$.

We classify these bad events according to the set $J\in {[Cn]\choose j}$ for which a saviour edge for $u$ 
was occupied by \RM exactly in rounds $t_u+h$, $h\in J$, and apply the union bound: 
\begin{align*} q_j \leq &  \sum_{J\in {[Cn]\choose j}} \prod_{h\in J} \left(1-p_{h}\right)\prod_{i\in [Cn]\setminus J} p_i \\
\leq & {Cn \choose j} \left(1- e^{-\frac{2a}{n}(1+o(1))} \right)^j \prod_{i\in [Cn]\setminus [j]} e^{-\frac{2a}{n}\left(1-3\alpha - \frac{i}{n}\right)} \\
\leq &  n^j \left(\frac{2a}{n}(1+o(1))\right)^j e^{- \frac{2a}{n}((1-3\alpha) (Cn -j)  - \sum_{i\in [Cn]\setminus [j]} \frac{i}{n})} \\
\leq & O((\ln \ln n)^j) e^{-2a (C (1-3\alpha) - \frac{C^2}{2} +o(1))} 
\end{align*}
In the second line we use that $e^{-\frac{2a}{n}\left(1-3\alpha - \frac{i}{n}\right)}$ is monotone increasing in $i$.
For (\ref{lemKOutgoingFastEarly}), we choose $C=1-3\alpha$ and obtain
\begin{align*}
q_j&\leq  O\left(\left(\ln{\ln{n}}\right)^j e^{-a((1-3\alpha)^2 +o(1))}\right)\\
& \leq O\left(\left(\ln{\ln{n}}\right)^j\ln^{-(1+\epsilon)((1-3\alpha)^2 +o(1))}n\right)\\
&=O\left(\ln^{-1-\frac{\epsilon}{2}}n\right). 
\end{align*}

For (\ref{lemKOutgoingFastLate}), we choose $C=\alpha$ and obtain
\begin{align*}
q_j&\leq O\left(\left(\ln{\ln{n}}\right)^j \ln^{-2(1+\epsilon)\alpha(1-\frac{7\alpha}{2}) +o(1)}n\right)\\
&=o(1)
\end{align*}
As $\ell < k$ is constant,
summing over these estimates for $j=0,1\dots,\ell -1$ gives the result in both cases (\ref{lemKOutgoingFastEarly})
and (\ref{lemKOutgoingFastLate}).
\end{proof}

\begin{cor}\label{corMinDegBad} For every $\epsilon>0$ there exists an $\alpha >0$, such that for every $k$ and 
every strategy $S$ of \CB the following holds a.a.s.
In the $(a:1)$-biased \RM-\CB game with $a=(1+\epsilon)\ln\ln n$ and \CB playing with strategy $S$, we have $d^+_{D_k}(u) =k$ 
for every $\alpha$-bad vertex $u$ by the end of round $t$.
\end{cor}
\begin{proof}
Recall that every bad vertex is an early or late candidate. 

By Lemma~\ref{LemKOutgoingFast}(i) the probability that any early candidate vertex does not have out-degree $k$
in $D_k$ by round $t- 2\alpha n$ is $o(\ln^{-1}n)$. 
Since there are at most $\frac{2t}{\alpha n}=\ln n$ early candidates, the union bound gives
that a.a.s. all early candidates have out-degree $k$ by round $t$.

By Lemma~\ref{LemKOutgoingFast}(ii) the probability that a late candidate vertex does not have out-degree $k$
in $D_k$ by round $t$ is $o(1)$. Now we claim that the number of late candidate vertices is 
$O(1)$ and hence applying the union bound again we get that they all have out-degree $k$ by round $t$ a.a.s.
 Indeed, since each late candidate has at most degree $\alpha n$ in \CB's graph at round $t-(1-\alpha)n$,
 \CB needed to claim at least $\alpha n$ incident edges at each late candidate in the next $(1-2\alpha)n$ rounds. 
 Thus, there can be at most $2\frac{(1-2\alpha)}{\alpha} = O(1)$ late candidate vertices, as promised.
\end{proof}

\begin{cor}\label{cor:PathsOnB} For every $\epsilon>0$ there exists $\alpha >0$, such that 
for every strategy $S$ of \CB the following holds a.a.s.
In the $(a:1)$-biased \RM-\CB game with $a=(1+\epsilon)\ln\ln n$ and \CB playing with strategy $S$, 
there are vertex-disjoint paths $P_1,\dots,P_\kappa$ in
\RM's graph that cover all $\alpha$-bad vertices and have their start- and endpoints, 
and only these, among the $\alpha$-good vertices.
\end{cor}
\begin{proof}
Let us use Corollary~\ref{corMinDegBad} with $k=2$, so we can assume that every bad vertex  has out-degree $2$
in $D_2$.  We start at an arbitrary bad vertex $v$ having no in-degree (such a vertex exists, since $D_2$ is acyclic by Lemma~\ref{lemDAcyclic}). We follow both of its outgoing edges in $D_2$ to create two vertex disjoint directed paths from $v$.
If we reach a good vertex we stop and choose it as the endpoint of our path. 
Otherwise, i.e. if the reached vertex $v'$ is bad, then it has out-degree $2$ in $D_2$, and we continue along 
one of the out-going edges. Since $D_2$ is acyclic and the number of bad vertices is finite, 
we must reach a good vertex eventually. 
Once both directed paths from $v$ are completed, their union in the underlying undirected graph of \RM 
forms a path $P_1$ with good endpoints and bad interior vertices. We remove the vertices of $P_1$ 
from $D_2$ and continue iteratively with a bad vertex that does not have an incoming edge, until there are no bad
vertices left. Note that, crucially, after the iterative removal of such rooted paths, 
all remaining vertices still have all their out-going edges, 
hence all remaining bad vertices still have out-degree $2$. Indeed, all vertices have in-degree at most $1$ and those 
with in-degree exactly $1$ that were removed also had their ancestor removed.
\end{proof}

\subsubsection{On the good vertices}
\label{sec:goodvertices}

Now that we have ``anchored" the bad vertices, let us turn to the good vertices. 
We show that the graph spanned by them is close enough to a truly random graph and 
make use of the strong expander properties of the latter.
To make these notions more precise, we switch to the point of view, where \RM's turns are determined by a random permutation $\sigma$. 

We consider the first $at$ random edges of $\sigma$ which surely were all ``tried'' to be
played by \RM in the first $t$ rounds. However he might not actually own each of these, because \CB
might have taken some of them by the time they were tried by \RM.
In the greatest generality, to be able to do multi-round exposure later, we consider subsets 
$M\subset [at]$ of coordinates of $\sigma$ and we will be interested in the truly random graph  
$\Gsigma{\sigma}{M}=\G{M}$ that consist of the edges exactly at these coordinates, that is, 
\begin{align*}
E\left(\G{M}\right)=\{\sigma(m) : m\in M \}.
\end{align*}
Note that the notion of $\Gsigma{\sigma}{[i]}$ coincides with the notion of $\Gsigma{\sigma}{i}$ defined earlier.

We define now a set of edges that will be ``forbidden'' for our analysis.
Recall that we fixed a strategy $S$ for \CB. Let $\HsigmaS{M} = \Hgraph{M}$ be the graph defined on the vertex set $[n]$ 
containing those edges $uv$ for which $uv \in \sigma(M)$ and for 
both $u$ and $v$ the edge $uv$ was among the first $3\alpha n$ incident edges which
\CB, playing according to $S$, claimed in the first $t$ rounds, when the permutation game 
according to $\sigma$ was played. 

The crucial point of this definition is the following simple lemma:

\begin{lem}\label{lemGoodBlocked}
Let $\sigma$ be an arbitrary permutation of the edges of $K_n$. Then for every subset $M\subseteq [at]$
the graph $\G{M} - E\left(\Hgraph{M}\right) - B$, with $B$ being the set of $\alpha$-bad vertices after $t$ rounds,
is a subgraph of \RM's graph.
\end{lem}
\begin{proof}
Let $uv$ be an edge of $\G{M} - E\left(\Hgraph{M}\right) - B$. Then $u$ and $v$ are both good vertices 
after $t$ rounds and hence  \CB's degree at both of them is at most $3\alpha n$. 
Thus, since $uv \in \sigma(M)$, if $uv$ would have been claimed by \CB 
up to round $t$ then $uv$ would be in $E\left(\Hgraph{M}\right)$. 
Consequently the edge $uv$ was not claimed by \CB in 
the first $t$ rounds. 
Now, since $uv\in \sigma(M) \subseteq \sigma([at])$ and \RM did try to claim the first at least $at$ edges of 
$\sigma$ in the first $t$ rounds, he must have claimed $uv$ by that time.
\end{proof}

The following lemma ensures that  not too many edges of cuts $(X,\overline{X}) : = \{xy : x\in X, y \in V\setminus X\}$
of $\G{[at]}$ are  ``blocked'' by \CB as one of its first $3\alpha n$ edges at the
endpoints. 
In particular, every vertex has small degree in $\Hgraph{[at]}$.

\begin{lem}\label{lemBlockedLocalUpperbound} The following is true a.a.s.
For every subset $X\subseteq [n]$,
we have that 
\begin{align*}
\left| E(\Hgraph{[at]}) \cap (X,\overline{X})\right| \leq \frac{8e\alpha at|X|}{n}
\end{align*}
\end{lem}
\begin{proof}
We write $H=\Hgraph{[at]}$. We create a random permutation $\sigma$ coordinate-wise. 
The crucial observation is that whether $\sigma(j) \in E(H)$ for some $j \in [at]$
depends only on the initial segment of the first $j-1$ edges of $\sigma$. 
Indeed, for  $\sigma(j)$ to be in $E(H)$, we need that at both endpoints it is one of the first 
$3\alpha n$ edges \CB claims when Maker plays according to $\sigma$. 
After Maker swiped through the first $j-1$ edges of $\sigma$, two things can happen:
either $\sigma(j)$ was taken by \CB in the game and hence was decided already whether 
it is one of the first $3\alpha n$ \CB-edges at both of its endpoints. 
If $\sigma(j)$ was not taken in the game, then Maker takes it in its next move and hence $\sigma(j)$ will not become 
part of $H$ later either. 

Hence, conditioning on any initial segment $\pi\in S_{E(K_n)}^{j-1}$, the
probability that the next edge $\sigma(j)$ is in $E(H)\cap\left(X,\overline{X}\right)$ depends only on
whether it is one of the at most $3\alpha n |X|$ edges that are already in $\Hgraph{[i-1]}$ and go
between $X$ and its complement. 
Furthermore, given that $\sigma$ starts with $\pi$, $\sigma(j)$ can take at least 
${n\choose 2}-at$ different values, each equally likely. 
Thus,
\begin{align*}
\prob{\sigma(j)\in E(H)\cap \left(X,\overline{X}\right) \given  \sigma |_{[j-1]} = \pi } \leq \frac{3\alpha n|X|}{{n\choose 2}-at}\leq\frac{7\alpha |X|}{n},
\end{align*}
for large $n$.
For our main estimate we can classify according to the set $L$ of coordinates where
the corresponding edges of $\sigma$ are from $E(H)\cap \left(X,\overline{X}\right)$ and apply the union bound:
\begin{align*}
\prob{E(H)\cap \left(X,\overline{X}\right) \geq\frac{8e\alpha at |X|}{n} }
&\leq\sum_{\substack{L\subset [at],\\ \size{L}=\frac{8e\alpha at|X|}{n}}} 
\prob{ \forall j\in L \, :\, \sigma(j)\in E(H)\cap \left(X,\overline{X}\right)}\\
&\leq {at\choose {\frac{8e\alpha at|X|}{n}}}\left(\frac{7\alpha |X|}{n}\right)^{\frac{8e\alpha at |X|}{n}}\\
&\leq \left(\frac{eat}{\frac{8e\alpha at |X|}{n}}\right)^{\frac{8e\alpha at |X|}{n}}\left(\frac{7\alpha |X|}{n}\right)^{\frac{8e\alpha at |X|}{n}}\\
&= \left(\frac{7}{8}\right)^{\frac{8e\alpha at |X|}{n}}
\end{align*}
Taking the union bound over all cuts $(X,\overline{X})$, we see that 
\begin{align*}
\sum_{s=1}^{n/2}{n\choose s} \left(\frac{7}{8}\right)^{\frac{8e \alpha at s}{n}}\leq 
\sum_{s=1}^{n/2}\left( n\left(\frac{7}{8}\right)^{4e\alpha^2\ln{n}\ln{\ln{n}} }\right)^s=o(1).
\end{align*}
\end{proof}

We also need the following standard fact from random graph theory; for completeness we include a proof in the Appendix.

\begin{lem}\label{lem:EdgesToComplement}
For all $\delta>0$, the following holds a.a.s in the random graph $G(n,m)=G$ with $m=\delta n\ln{n}\ln{\ln{n}}$.  
For every vertex set $X\subset [n]$ of size $|X| \leq \frac{n}{2}$, we have $$E(G)\cap\left(X,\overline{X}\right)\geq \size{X}\frac{m}{2n}.$$
\end{lem}

\subsubsection{\CM builds a connected graph and achieves a large minimum degree}

We now have all the necessary tools to conclude the theorems about the 
min-degree $k$ game and the connectivity game.

\begin{proof}[Proof of Theorems~\ref{thm:RMCB-mindegree-k} and \ref{thm:RMCB-connectivity}] 
By Theorem~\ref{thm:CBWinDegree1} here we need to take care of the upper bounds only.

Let $\alpha < \frac{1}{32e}$ be arbitrary such that \eqref{eq:alpha} is satisfied.
Define $\delta =(1+\epsilon)\frac{\alpha}{2}$, so $at = m = \delta n \ln n \ln\ln n$.
We show that by round $t$ \RM's graph is connected and has minimum degree at least $k$ a.a.s.

Recall that by Corollary~\ref{corMinDegBad} and Lemma~\ref{lemDAcyclic} all 
bad vertices have degree at least $k$ in \RM's graph by round $t$ a.a.s.
Moreover, by Corollary~\ref{cor:PathsOnB} and Lemma~\ref{lemDAcyclic} 
every bad vertex is connected to some good vertex 
via a path in \RM's graph a.a.s. 

It is enough to show that \RM's graph induced by the 
set of good vertices is connected and has minimum degree at least $k$.
We will use Lemma~\ref{lemGoodBlocked}.

Let $X\subseteq [n]$ be an arbitrary subset of good vertices, of size $|X|\leq \frac{n}{2}$. 
By Lemma~\ref{lem:EdgesToComplement} there are at least $|X|\frac{at}{2n}$ edges in $\G{[at]}$
between $X$ and $\overline{X}$.  
At most $\frac{8e\alpha at |X|}{n}$ of these $\frac{at |X|}{2n}$ edges  are in $\Hgraph{[at]}$ by 
Lemma~\ref{lemBlockedLocalUpperbound}, and 
at most another $|X|\ln n$ of them are going to an $\alpha$-bad vertex (since there are at most $\ln n$ bad vertices).

The rest of these edges is in \RM's graph by  Lemma~\ref{lemGoodBlocked}. 
That means that at least $\left(\left(\frac{1}{2} - 8 e \alpha \right)\frac{at}{n} - \ln n\right)|X| = \Omega (|X|\ln n \ln \ln n )\geq k$ edges
of \RM's graph leave $X$ to its complement among the good vertices. 
In particular, each good vertex $v$ has degree at least $k$ in \RM's graph.
\end{proof}

\subsection{\RM builds a Hamilton cycle} \label{sec:RMHamiltonicity}

We now turn to the Hamiltonicity game. The plan is the following: We use Corollary \ref{cor:PathsOnB} to find paths covering the bad vertices. Then we connect them to one long path, using short paths on the good vertices. Finally, we show that the rest is Hamilton connected, which allows us to close the loop using all remaining vertices. To find the short paths and prove Hamilton connectivity, we turn away from the game for a while, and look at random graphs in general.

\subsubsection{Short Paths}

The following precise notion of expansion from~\cite{krivelevich2014robust} will be central to our proofs. 
Here $N(X)$ denotes the set of vertices which have a neighbour in $X$. 

\begin{defin}\label{def:halfExpander}
	Let $\lambda$ and $r$ be positive reals. A graph $G$ is a \textit{half-expander with parameters $\lambda$ and $r$} if the following properties hold:
	\begin{enumerate}
		\item For every set $X$ of vertices of size $\left|X\right|\leq\frac{\lambda n}{r}$, $\left|N(X)\right|\geq r\left|X\right|$,
		\item for every set $X$ of vertices of size $\left|X\right|\geq\frac{n}{\lambda r}$,  $\left|N(X)\right|\geq \left(\frac{1}{2}-\lambda\right)n$, and
		\item for every pair of disjoint sets $X$,$Y$ such that $\left|X\right|,\left|Y\right|\geq\left(\frac{1}{2}-\lambda^{1/5}\right)n$, $e(X,Y)>2n$.
	\end{enumerate}
\end{defin}

The following tail estimates for the {\emph hypergeometric distribution} will be very convenient. 
Let $F$, $f$ and $l$ be positive integers such that $f,l\leq F$. The value of the random variable $X$ is the size of the intersection of fixed $f$-element subset 
$M \subseteq [F]$ with a uniformly chosen $l$-subset $M^*$. Note that the expected value of $X$ is $\frac{fl}{F}$. 
For the following standard estimates see e.g. \cite{janson2011random} Theorem 2.10.

\begin{thm}\label{thm:hypergeometric}
Let $X$ have the hypergeometric distribution with parameters $F$, $f$ and $l$. Then
\begin{align}
\prob{X\geq 2\frac{fl}{F}}&\leq e^{-\frac{fl}{3F}}\\
\prob{X\leq \frac{fl}{2F}}&\leq e^{-\frac{fl}{8F}}.
\end{align}
\end{thm}
We will use the theorem to estimate how many edges of a "good" edge set of size $f$ are realized in $G(n,m)$. 

The following useful properties of the random graph are consequences of Theorem~\ref{thm:hypergeometric}; a proof is included in the Appendix.
\begin{lem}\label{lem:smallSetsHaveLargeNeighborhood}
For all constant $\delta>0$ let $m=\delta n\ln{n}\ln{\ln{n}}$ and $G=G(n,m)$, then the following the following three properties hold with probability 
at least $1-e^{-\Omega(\ln n\ln \ln n)}$.
\begin{itemize}
\item[(a)] Every vertex set $X$ of size at most $\size{X}\leq \frac{n^2}{m}$ has a neighborhood of size at least $\size{N_G(X)}\geq \size{X}\frac{m}{8n}$.
\item[(b)] for every pair of vertex sets $X\subset [n]$ of size $\frac{n}{4}\geq |X| \geq \frac{64n^2}{m}$, 
and $N\subset [n]$ of size $\size{N}\leq \frac{n}{2}$ there are at least $\size{X}\frac{m}{8n}$ edges between $X$ and $[n]\setminus (X\cup N)$ in $G(n,m)$.
\item[(c)] for every pair of disjoint vertex sets $X,Y\subset [n]$ of size at least $\size{X},\size{Y}\geq \frac{n}{4}$ there are at least $m\size{X}/8 n$ edges between $X$ and $Y$ in $G(n,m)$.
\end{itemize}
\end{lem}

First we show that random graphs are half-expanders, with some resilience to edge and vertex removal. 
This will be useful in particular with respect to Lemma \ref{lemBlockedLocalUpperbound}. 
We state the lemma in a bit more general form than is need in this section in order to provide us with some leeway later.

\begin{lem}\label{lem:halfExpander_GM_D}
For all $0<\lambda<2^{-11}$ and all $\delta>0$, the following holds: Let $m=\delta n\ln{n}\ln{\ln{n}}$ and $G=G(n,m)$. Further let $\mathcal{D} \subseteq {{[n]}\choose {\leq \ln^2 n}}$ 
be a family of $n^{3\ln{n}}$ vertex subsets such that each set $D\in\mathcal{D}$ has size at most $\size{D}\leq \ln^2{n}$. Then with probability at least $1-e^{-\Omega(\ln n \ln{\ln{n}})}$, for all $D\in\mathcal{D}$ and all graphs $H\subset K_n$ with maximum degree at most $\Delta(H)\leq \frac{m}{32n}$, the graph $G-E(H)-D$ is a half-expander with parameters $\lambda$ and 
$r=\frac{m}{16n} = \frac{\delta}{16} \ln n\ln\ln n$.
\end{lem}
\begin{proof}

We first show the following.\\[2mm]
\textbf{Claim.}
{\em With probability at least $1-e^{-\Omega(\ln^2 n\ln \ln n)}$, for all $D\in\mathcal{D}$ and all $v\in [n]$, $v$ has at most $\frac{m}{32n}$ $G$-neighbors in $D$.}
\begin{proof} 
For a fixed vertex $v$, set $D\in {\cal D}$ and subset $Q\subseteq D$ of size $\size{Q}=q = \frac{m}{32n}$, the probability that all vertices in $Q$ are $G$-neighbors of 
$v$ is $\frac{{N-q \choose m-q}}{{N\choose m}} = \prod_{i=0}^{q-1} \frac{m-i}{N-i} \leq \left( \frac{m}{N}\right)^q$, where $N={n\choose 2}$.\\
Taking the union bound over all $v, D$, and $Q$, yields that the failure probability of the event in the claim is at most
$$n|{\cal D}|{\size{D}\choose q}\left( \frac{m}{N}\right)^q\leq n\cdot n^{3\ln n} \left(\frac{200\ln^2 n}{n} \right)^q = n^{-\Omega\left(\ln{n}\ln{\ln{n}}\right)}$$
and the claim is proved.
\end{proof}
Since  the events in the claim and Lemma \ref{lem:smallSetsHaveLargeNeighborhood} hold with probability at least $1-e^{-\Omega(\ln n \ln \ln n)}$,
it is enough to show that they imply the event in our lemma.
 
Let $D\in {\cal D}$ be an arbitrary set from the family ${\cal D}$ and let $H$ be an arbitrary graph with maximum degree $\Delta (H) \leq \frac{m}{32n}$.
First note that by the property of the claim, removing $D$ and $E(H)$ from $G$ removes at most $\frac{m}{16n}$ incident edges at any vertex in $V(G)\setminus D$.

To show the first property of Definition \ref{def:halfExpander}, fix $X\subset V(G)\setminus D$ such that $\size{X}\leq \frac{\lambda (n-\size{D})}{r}$.
Note that then $\size{X}\leq \frac{n^2}{m}$, so by Lemma \ref{lem:smallSetsHaveLargeNeighborhood}(a) the neighborhood of $X$ in $G$ 
has size at least $\size{N_G(X)}\geq \size{X}\frac{m}{8n}$.
Removing $D$ and $E(H)$ eliminates at most $\frac{\size{X}m}{16 n}$ edges incident to $X$, which means that after the removal, 
the neighborhood of $X$ has  size at least  $\frac{\size{X}m}{16 n}= r\size{X}$.

For the second property, let us fix a set $X$ of size $\size{X}=\frac{n-\size{D}}{\lambda r}$.
Note that $\frac{64 n^2}{m}\leq\size{X} \leq \frac{n}{4}$. Assume that the neighborhood $N$ of $X$ in $G-E(H)-D$ has size less than 
$(\frac{1}{2} - \lambda)(n - |D|)$. Then by the property in Lemma~\ref{lem:smallSetsHaveLargeNeighborhood}(b),
there are at least $\size{X}\frac{m}{8n}$ edges between $X$ and $[n]\setminus{\left(X\cup N\right)}$ in $G$. 
Removing $D$ and $E(H)$ removes at most $\size{X}\frac{m}{16n}$ edges. Thus, there is an edge from 
$X$ to outside of $N$ in $G-E(H)-D$, a contradiction to the definition of $N$.

For the third property, fix two disjoint vertex sets $X,Y\subset [n]\setminus D$ of size at least $\left(\frac{1}{2}-\lambda^{1/5}\right)(n-\size{D})$. Note that 
$|X|, |Y|  \geq \frac{n}{4}$ for $n$ large enough, so we can apply Lemma~\ref{lem:smallSetsHaveLargeNeighborhood}(c) and conclude that 
there are at least $\frac{m\size{X}}{8n}$ edges between $X$ and $Y$ in $G$. Therefore, at least $\left(\frac{1}{8} - \frac{1}{16}\right)\size{X}\frac{m}{n}\geq 2n$ edges remain after removing $D$ and $E(H)$, for $n$ sufficiently large.
\end{proof}

Since we now know we are working with a half-expander, we can do the first step towards Hamiltonicity by connecting vertices with short paths. 

\begin{thm}\label{thmShortPaths}
There is a $\lambda_0>0$ such that for all $\lambda<\lambda_0$, the following holds: Let $G$ be a half-expander on $n$ vertices with parameters $\lambda$ and 
$r \geq \frac{8}{\lambda^2}\ln n$, and let $k\leq \ln{n}$. 
Then for all pairwise distinct points $a_1,\dots,a_k,b_1,\dots,b_k$, there are vertex disjoint paths $P_1,\dots,P_k$, 
each of length at most $\ln{n}$, such that $P_i$ connects $a_i$ to $b_i$.
\end{thm}
\begin{proof}
We build the paths simultaneously, starting at both ends and keeping sets of possible vertices at the different positions in the paths, from which we then can choose to connect the two partial paths we built. Throughout the proof, let $q:=\frac{r}{8\ln{n}} \geq \frac{1}{\lambda^2}$. Note that $q\geq 2$ for $\lambda_0$ sufficiently small.\\
Let $j_0=\left\lceil\frac{\ln{\frac{\lambda n}{r}}}{\ln{q}}\right\rceil$. For $0\leq j\leq j_0+1$ and $1\leq i\leq k$ we will define 
vertex sets $D^+_{i,j}$ and $D^-_{i,j}$, such that $D^+_{i,0}=\{a_i\}$ and $D^-_{i,0}=\{b_i\}$, all the $2(j_0+2)k$ sets $D^+_{i,j}$ and $D^-_{i,j}$ are
pairwise disjoint, for every $i,j$ we have $D^+_{i,j}\subseteq N\left(D^+_{i,j-1}\right)$ and $D^-_{i,j}\subseteq N\left(D^-_{i,j-1}\right)$, and $\size{D^-_{i,j}}= \size{D^+_{i,j}}=f(j)$ where
\begin{align*}
f(j)= \begin{cases} q^j&\text{if }j<j_0\\
\frac{\lambda n}{r}&\text{if }j=j_0\\
\frac{n}{\lambda r}&\text{if }j=j_0+1.\end{cases}
\end{align*}
We define the sets iteratively over $j$, where in each step, we iterate over $i$.\\
First, let $D^+_{i,0}:=\{a_i\}$ and $D^-_{i,0}:=\{b_i\}$ for $i=1,\dots,k$.\\
Now let us fix $1\leq j\leq j_0+1$ and $1\leq i\leq k$, and assume that for all $j'<j$ and all $1\leq i'\leq i$, the sets $D^+_{i',j'}$ and $D^-_{i',j'}$ are constructed, and for all $i''<i$, the sets $D^+_{i'',j}$  and $D^-_{i'',j}$ are constructed.\\
We first define
\begin{align*}
A^\pm_{i,j}:=N\left(D^\pm_{i,j-1}\right)\setminus\left(\left(\bigcup_{i''<i}D^+_{i'',j}\cup D^-_{i'',j}\right)\cup\left(\bigcup_{1\leq i'\leq k,j'<j}D^+_{i',j'}\cup D^-_{i',j'}\right)\right).
\end{align*}
We show that we can find the $D^\pm_{i,j} \subseteq A^\pm_{i,j}$ with the required properties by proving
\begin{align*}
\size{A^\pm_{i,j}}\geq f(j).
\end{align*}

Let us first consider the case $j\leq j_0$. Then for all $j'<j$ and $1\leq i'\leq k$, we have $\size{D^\pm_{i',j'}}=f(j')=q^{j'}$. Further, since $G$ is a half-expander and $\size{D^\pm_{i,j-1}}\leq \frac{\lambda n}{r}$, we have $\size{N\left(D^\pm_{i,j-1}\right)}\geq r\size{D^\pm_{i,j-1}}= (8\ln n)q^j$. Finally note that $\size{D^\pm_{i'',j}}\leq q^j$ for all $i''<i$ (this holds also if $j=j_0$, since $f(j_0)=\frac{\lambda n}{r}\leq q^{j_0}$). Therefore,
\begin{align*}
\size{A^\pm_{i,j}}&\geq \size{N(D^\pm_{i,j-1})}- \sum_{i''=1}^{i-1}\size{D^+_{i'',j} \cup D^-_{i'',j}}-\sum_{j'=0}^{j-1}\sum_{i'=1}^{k}\size{D^+_{i',j'} \cup D^-_{i',j'}}\\
& \geq (8\ln n)q^j - (\ln{n}-1)(2q^j) - \ln n \sum_{j'=0}^{j-1} 2q^{j'}\\
& \geq (4\ln n +2)q^j \geq q^j \geq f(j)
\end{align*}
where we used that $i\leq k\leq \ln{n}$, and $ \sum_{j'=0}^{j-1} q^{j'} \leq q^j$ as $q\geq 2$. 

In the case $j=j_0+1$, note that
\begin{align} \label{eq:j_0}
\size{\bigcup_{1\leq i'\leq k}D^+_{i',j_0}\cup D^-_{i',j_0}}=2k\frac{\lambda n}{r}
\end{align}
and 
\begin{align}\label{eq:j}
\size{\bigcup_{1\leq i'\leq k, j'<j_0}D^+_{i',j'}\cup D^-_{i',j'}}= 2k\sum_{j'=0}^{j_0-1}q^{j'}\leq  4kq^{j_0-1}\leq 4k\frac{\lambda n}{r},
\end{align}
by the definition of $j_0$. Again using the half-expander property of $G$, we have that
\begin{align*}
\size{A^\pm_{i,j_0+1}}&\geq r\frac{\lambda n}{r} -2(i-1)\frac{n}{\lambda r}- \left(2k\frac{\lambda n}{r}+4k\frac{\lambda n}{r}\right)\\
&\geq\left(\lambda^2 r-2(\ln n-1)-6\lambda^2 \ln n\right)\frac{n}{ \lambda r}\\
&\geq\frac{n}{\lambda r}
\end{align*}
using that $r=8q\ln{n}$, $q\geq\frac{1}{\lambda^2}$ and $i\leq k\leq \ln{n}$.
This concludes the proof that we can construct the sets $D^\pm_{i,j}$ with the properties described above. We now find paths for all $i$, using the $D^\pm_{i,j}$.\\
Suppose we have constructed appropriate paths $P_1, \ldots , P_{i-1}$ already.
To construct $P_1$, let us first define
\begin{align*}
D^\pm_{i,j_0+2}=N\left(D^\pm_{i,j_0+1}\right)\setminus \left(\bigcup_{\substack{1\leq i'\leq k\\0\leq j\leq j_0+1}}\left( D^+_{i',j} \cup D^-_{i',j}\right) \cup \bigcup_{i'' =1}^{i-1} V(P_{i''})\right).
\end{align*}
Since $|D^\pm_{i,j_0+1}| = \frac{n}{\lambda r}$ we can use the second half-expander property for $G$. This, together with the 
estimates \eqref{eq:j_0}, \eqref{eq:j}, and $\frac{n}{r} \leq \frac{\lambda^2 n}{8\ln n}$ implies 
$$\size{D^\pm_{i,j_0+2}} \geq \left( \frac{1}{2} -\lambda \right) n - \ln n \left( 2\frac{n}{\lambda r} + 2\frac{\lambda n}{r} + 4 \frac{\lambda n}{r} + \ln n\right)  \geq\left(\frac{1}{2}-\lambda^{1/5}\right)n,$$ for $\lambda$ small enough. 
If the sets $D^+_{i,j_0+2}$ and $D^-_{i,j_0+2}$ are disjoint, then using the third half-expander property we can conclude that 
there is an edge $e$ between them. Retracing a path from each endpoint of $e$ through the $D^\pm_{i,j}$ back to $D^+_{i,0}=\{a_i\}$ and $D^-_{i,0}=\{b_i\}$, respectively,
and concatenating them with $e$ gives us the required $a_i, b_i$-path $P_i$. The length of $P_i$ then is $j_0+3\leq\ln{n}$, indeed. If $D^+_{i,j_0+2}$ and $D^-_{i,j_0+2}$ are not disjoint, we can trace back a path to $a_i$ and $b_i$
from any vertex in the intersection and then $P_i$ is of length $j_0 +2$.
\end{proof}

The next corollary is a direct consequence of Lemma \ref{lem:halfExpander_GM_D} and Theorem \ref{thmShortPaths}.
\begin{cor}\label{cor:GnmShortPaths}
For all $\delta>0$ the following holds 
in the random graph $G(n, m)$ with $m=\delta n\ln{n}\ln{\ln{n}}$ a.a.s.
For all vertex sets $B$ with $\size{B}\leq \ln{n}$, all sequences of
pairwise distinct points $a_1,\dots,a_k,b_1,\dots,b_k\in V\setminus B$, $k\leq \ln n$, and all graphs $H$ with $\Delta(H)\leq \frac{m}{32n}$, there are vertex disjoint 
paths $P_1,\dots,P_{k-1}$ in $G(n,m) -E(H)- B -\{ a_1, b_k\}$, each of length at most $\ln{n}$, 
such that $P_i$ connects $a_{i+1}$ to $b_i$.
\end{cor}
\begin{proof}
Let $\lambda>0$ be small enough such that Lemma \ref{lem:halfExpander_GM_D} and Theorem \ref{thmShortPaths} both hold. Further let $\mathcal{D}$ be the family of all vertex sets $B \cup \{ a_1, b_k\}$ where $\size{B}\leq\ln n$. Note that $|{\cal D}| = {n\choose \ln n} \leq n^{\ln n}$.  
Then by Lemma \ref{lem:halfExpander_GM_D}, a.a.s. for every $B\in\mathcal{D}$ and every graph $H\subset K_n$ with $\Delta(H)\leq \frac{m}{32 n}$, the graph $G(n,m)-E(H)-B- \{ a_1, b_k\}$ is a half-expander with parameters $\lambda$ and $r=\frac{m}{16n}=\frac{\delta}{16}\ln n\ln \ln n\geq\frac{8}{\lambda^2}\ln n$. Applying Theorem \ref{thmShortPaths} to these graphs concludes the proof.
\end{proof}

\subsubsection{Hamilton Connectivity}

We now turn towards Hamilton connectivity. This section relies heavily on the works of Lee and Sudakov~\cite{LeeSudakov2011} and Krivelevich, Lee, and 
Sudakov~\cite{krivelevich2014robust}. 
The following properties prove to be a valuable criterion for Hamiltonicity.

\begin{defin}\label{defRE}
Let $\xi$ be a positive constant. We say that a graph $G$ has property $\mathcal{RE}\left(\xi\right)$ if it is connected, and for every path $P$ with a fixed edge $e$, (i) there exists a path containing $e$ longer than $P$ in the graph $G\cup P$, or (ii) there exists a set of vertices $S_P$ of size $\left|S_P\right|\geq\xi n$ such that for every vertex $v\in S_P$, there exists a set $T_v$ of size $\left|T_v\right|\geq\xi n$ such that for every $w\in T_v$, there exists a path containing $e$ of the same length as $P$ that starts at $v$, and ends at $w$.
\end{defin}

\begin{defin}\label{defComplementsRE}
Let $\xi$ be a positive constant and let $G_1$ be a graph with property $\mathcal{RE}\left(\xi\right)$. We say that a graph $G_2$ \textit{complements} $G_1$, if for every path $P$ with a fixed edge $e$, (i) there exists a path containing $e$ longer than $P$ in the graph $G_1\cup P$, or (ii) there exist $v\in S_P$ and $w\in T_v$, such that $\{v,w\}$ is an edge of $G_1\cup G_2\cup P$ (the sets $S_P$ and $T_v$ are as defined in Definition\ref{defRE}).
\end{defin}

\begin{prop}[\text{\cite[Proposition 3.3]{krivelevich2014robust}}] \label{propHamConnected}
Let $\xi$ be a positive constant. If $G_1\in\mathcal{RE}\left(\xi\right)$ and $G_2$ complements $G_1$, then $G_1\cup G_2$ is Hamilton connected.
\end{prop}

Again, the notion of a half-expander comes in useful.
\begin{lem}[\text{\cite[Lemma 3.5]{krivelevich2014robust}}]\label{lemPropertyRE}
There exists a positive $\lambda_0$ such that for every positive $\lambda\leq \lambda_0$, the following holds for every $r\geq 16\lambda^{-3}\ln{n}$: every half-expander on $n$ vertices with parameters $\lambda$ and $r$ has property $\mathcal{RE}\left(\frac{1}{2}+\lambda\right)$.
\end{lem}

The next lemma and its proof are based on \cite{LeeSudakov2011} and adapted to our situation.

\begin{lem}\label{lemComplements}
For all $0<\lambda\leq 1/2$ there is a $\beta>0$ such that for all $\delta>0$,
for the uniform random graph $G=G(n,m)$ with 
$m=\delta n\ln{n}\ln{\ln{n}}$ edges, the following holds with probability at least 
$1-e^{-\Omega(m)}$:\\ 
For every graph $H$ with maximum degree $\Delta(H)\leq \frac{m}{8n}$, 
the graph $G-E(H)$ complements every subgraph $R\subseteq G$ with property 
$\mathcal{RE}(\frac{1}{2}+\lambda)$ that has at most $\beta m$ edges.
\end{lem}
\begin{proof}

Let us fix a graph $R\subseteq K_n$ with at most $\beta m$ edges such that $R\in \mathcal{RE}\left( \frac{1}{2} +\lambda\right)$.  
We will estimate the probability that $R\subseteq G$ and there exists an $H$ with maximum degree $\Delta(H)\leq \frac{m}{8n}$, such that
the graph $G-E(H)$ does not complement $R$. 

For this we fix a path $P$ and an edge $e\in E(P)$ and estimate from above 
the probability that there exists an $H$ with maximum degree $\Delta(H)\leq \frac{m}{8n}$, such that
(i) in $R\cup P$ no path containing $e$ is longer than $P$ and (ii) for every $v\in S_P$ and every $w\in T_v$ we have $vw \not\in E((G-E(H)) \cup R \cup P)$
(where $S_P$ is the set of size $\size{S_P}\geq \left(\frac{1}{2}+\lambda \right)n$ and $T_v$ the set of size $\size{T_v}\geq \left(\frac{1}{2}+\lambda \right)n$ 
from Definition~\ref{defRE} applied to $R$).
Observe that (i) is not a random statement, hence we can assume that it holds for $P$ and $e$, otherwise the probability is $0$.

Note also that if there exists a vertex $v\in S_P$ and a $w\in T_v$ such that $\{v,w\} \in E(R)$, then the probability is $0$ as well. 
Thus from now on we also assume that for all $v\in S_P$ and all $w\in T_v$, the edge $\{v,w\} \not\in E(R)$.

Let now $S_P'=\left\{v_1,\dots,v_{\lambda n}\right\}$ be an arbitrary subset of $S_P$ of size $\size{S_P'}=\lambda n$. For all $v\in S_P'$, let $T_v'$ be a subset of $T_v\setminus S_P'$ of size $\size{T_v'}=\frac{1}{2}n$. Note that the edge sets $E_v'=\left\{\left\{v,w\right\}:w\in T_v'\right\}$ for $v\in S_P'$ are all disjoint, since $S_P'$ is disjoint from every $T_v'$. Hence their union 
$$E':=\left\{\{v,w\}:v\in S_P', w\in T_v'\right\}$$
has size $|E'|= \frac{\lambda}{2}n^2$.

We will show that with high probability, for every $H$ with $\Delta(H)\leq \frac{m}{8n}$, there is a $v\in S_P'$ and a $w\in T_v'$ such that $\{v,w\}$ is an edge of $G-E(H)$. 
For that, it is sufficient that, independently of $H$, there are at least $\lambda m/4$ edges in $E(G)\cap E'$.
Indeed, removing the edges  of any graph $H$ with maximum degree $\Delta (H) \leq \frac{m}{8n}$ can eliminate at most 
$\size{S_P'}\frac{m}{8n}\leq \frac{\lambda m}{8}$ edges from $E(G)\cap E'$, 
which means that at least $\frac{\lambda m}{8}>0$ edges of $E'$ are left in $(E(G)\setminus E(H))\cap E'$.

Recall that we assumed that $E(R)$ is disjoint from $E'$, but condition on $E(R)\subseteq E(G)$. 
Thus, the size of $E(G)\cap E'$ has a hypergeometric distribution with parameters $F={n\choose 2}\setminus \size{E(R)}\leq n^2/2$, 
$f=\size{E'}=\lambda n^2/2$ and $l=m-\size{E(R)}\geq m/2$. Hence for the expectation we have 
$\frac{fl}{F} \geq \frac{\lambda m}{2}$ and then Theorem \ref{thm:hypergeometric} implies that
\begin{align*}
\prob{\size{E(G)\cap E'}\leq \lambda m/4\given R\subset G}\leq e^{-\lambda m/16}.
\end{align*}
Taking the union bound for all choices of $P$ and $e\in E(P)$ we obtain that 
\begin{align*}
\prob{G\text{ does not complement } R\given R\subset G}\leq n n! e^{-\lambda m/16}.
\end{align*}
Finally, taking the union bound for all $R\subseteq K_n$ with at $k \leq \beta m$ edges and using that 
$\prob{R\subset G}
\leq (\frac{m}{{n\choose 2}})^k,$
we obtain that our failure probability is at most
\begin{align*}
& \sum_{R\in\mathcal{RE}(\frac{1}{2}+\lambda), \size{E(R)}\leq \beta m} \prob{G\text{ does not complement }R\given R\subset G} \prob{R\subset G}\\
&\leq n n! e^{-\lambda m/4}\sum_{k=1}^{\beta m}{{n\choose 2}\choose k}\left(\frac{m}{{n\choose 2}}\right)^k
\leq e^{-\lambda m/5} \sum_{k=1}^{\beta m}\left(\frac{em}{k}\right)^k\\
& \leq  e^{-\lambda m/5}  \beta m\left(\frac{e}{\beta}\right)^{\beta m}.
\end{align*}
Here we used that the terms of the last sum are monotone increasing for $k\leq \beta m$, as long as $\beta <1$. 
Thus the event of the lemma fails  with probability $e^{-\Omega(m)}$ provided $\beta$ is sufficiently small.
\end{proof}

The next statement wraps up this section.
\begin{cor}\label{cor:GnmHamConnected}
There is a $\gamma>0$ such that for every $\delta >0$ and every family $\mathcal{D} \subseteq {[n]\choose \leq \ln^2 n}$ of at most $n^{3\ln{n}}$ vertex subsets of size at most $\ln^2{n}$ each,  the following holds with probability at least $1-e^{-\Omega(\ln n \ln \ln n)}$:\\
For every $D\in\mathcal{D}$ and every graph $H\subset K_n$ with $\Delta(H)\leq \gamma\frac{m}{n}$, the graph $G(n,m)- E(H)-D$, 
with $m=\delta n\ln{n}\ln{\ln{n}}$, is Hamilton connected.
\end{cor}
\begin{proof}
Let $\lambda_0$ as in Lemma \ref{lemPropertyRE} and $0<\lambda<\min\left(\lambda_0,2^{-11}\right)$. Let $0<\beta<1$ such that Lemma \ref{lemComplements} holds and let $\gamma=\frac{\beta}{32}$. Let $\mathcal{D} \subseteq {[n]\choose \leq \ln^2 n}$
be a family of at most $n^{3\ln{n}}$ vertex subsets of size at most $\ln^2{n}$ each. Then let $G\sim G(n,m)$ be a random graph and
let $G'$ be a uniformly random subgraph of $G$ with $\frac{\beta m}{2}$ edges. 
Let ${\cal E}$ be the event that for every $D\in\mathcal{D}$ and $H\subseteq K_n$ with $\Delta(H)\leq \gamma\frac{m}{n}$, the graph $G'-E(H)-D$ is a half-expander with parameters $\lambda$ and $r=\frac{\beta m}{32n}$. By definition, $G'$ is distributed like $G\left(n,\frac{\beta m}{2}\right)$ and thus by Lemma \ref{lem:halfExpander_GM_D}, ${\cal E}$ holds with probability at least $1-e^{-\Omega(\ln n \ln \ln n)}$.

We now fix a $D\in\mathcal{D}$. Let ${\cal A}_D$ be the event that $G-D$ has at least $m/2$ edges.
We show that ${\cal A}_D$ fails with probability at most $e^{-\Omega(n)}$. 
Let $N:={n\choose 2}$. Note that there are at most $n\size{D}$ edges incident to $D$. 
Then the probability that removing $D$ from $G$ removes at least $k=\frac{1}{2} m$ edges is at most
\begin{align*}
\frac{{{n\size{D}}\choose{k}}{{N-k}\choose{m-k}}}{{N\choose m}}\leq \left(\frac{en\size{D}}{k}\right)^k\left(\frac{m}{N}\right)^k\leq \left(\frac{4e\ln^2{n}}{n-1}\right)^{k}\leq e^{-\Omega(m\ln n)}.
\end{align*}
From now on we condition on ${\cal A}_D$ holding.

Let ${\cal B}_D$ be the event that for every $H\subseteq K_{n-|D|}$ with $\Delta (H) \leq \frac{m}{16(n-|D|)}$, 
$G-E(H)-D$ complements every one of its subgraphs with at most $\frac{\beta m}{2}$ edges that 
has property $\mathcal{RE}\left(\frac{1}{2}+\lambda\right)$. 

Now we condition further on $G-D$ having exactly $k\geq m/2$ edges. Note that then $G-D$ is distributed as $G(n-\size{D},k)$. 
Under this condition then, by Lemma \ref{lemComplements}, ${\cal B}_D$ fails with probability at most $e^{-\Omega(m)}$. 
Since the events $\size{E(G-D)}=k$, $k\geq m/2$ partition the event ${\cal A}_D$, we obtain that 
$\prob{\overline {B_D} | {\cal A}_D}\leq e^{-\Omega(m)}$.
Hence, in total we get that 
$$\prob{ \bigcup_{D\in {\cal D}} \overline {B_D} } \leq |{\cal D}| \left(e^{-\Omega(m)}+e^{-\Omega(m\ln n)}\right) = e^{-\Omega(m)}.$$

Thus, with probability at least $1-e^{-\Omega(\ln n \ln \ln n)}$ both
the event ${\cal E}$ and the events ${\cal B}_D$ hold for every $D\in\mathcal{D}$.

If ${\cal E}$ holds, then by Lemma \ref{lemPropertyRE}, $G'-E(H)-D$ has property $\mathcal{RE}\left(\frac{1}{2}+\lambda\right)$ for every $D\in\mathcal{D}$ and $H\subseteq K_n$ with $\Delta(H)\leq \gamma\frac{m}{n}$.

If ${\cal B}_D$ holds, then $G-E(H)-D$ complements the $G'-E(H)-D$, because the latter has property $\mathcal{RE}\left(\frac{1}{2}+\lambda\right)$ and has at most $\frac{\beta m}{2}$ edges.
Thus, by Proposition \ref{propHamConnected}, $G-E(H)-D$ is Hamilton connected (recall that $G'$ is a subgraph of $G$). 
\end{proof}

\subsubsection{Proof of the Hamiltonicity threshold}

We are now ready to return to the Hamiltonicity game. 
\begin{proof}[Proof of upper bound in Theorem \ref{thm:RMCB-hamiltonicity}]
Let $\epsilon>0$ fixed and  let $a= (1+\epsilon)\ln{\ln{n}}$. Furthermore 
let \CB play according to an arbitrary fixed strategy $S$. 

Fix an $\alpha < \frac{1}{16 e}\min\{ \gamma_{(\ref{cor:GnmHamConnected})}, \frac{1}{32}\}$, such that inequality \eqref{eq:alpha} holds as well. 
Recall that $t=\frac{\alpha}{2}n\ln{n}$. 
We show that \RM builds a Hamilton cycle in the first $t$ rounds of the $(a:1)$-biased Hamiltonicity game a.a.s.
 
By Proposition~\ref{prop:permutationStrategy} we work in the setup where \RM plays according to  
a random permutation  $\sigma\in S_{E(K_n)}$  against \CB's fixed strategy $S$. 
To use our random graph statements we generate $\sigma$ in three steps.
First we select the initial segment $\sigma_1$ of the first $\frac{at}{2}$ edges of $\sigma$ uniformly at random.
Then, independently, we select another sequence $\sigma_2$ of $\frac{at}{2}$ edges uniformly at random from {\em all} 
$\frac{{n\choose 2}!}{\left( {n\choose 2} - \frac{at}{2} \right)!}$
 choices and append, in this order, those edges of $\sigma_2$ to $\sigma_1$ which do not appear in it already. 
Finally, we choose a uniformly random permutation $\sigma_3$ of the rest of the edges and append it, to obtain $\sigma$.
We define the set $M_2= M_2(\sigma_1, \sigma_2) \subseteq [at]$ to be the set of those coordinates where the edges of $\sigma_2$ appear in $\sigma$.

We thus refined the probability space to a triplet $\left(\sigma_1,\sigma_2,\sigma_3\right)$. But still, clearly the permutation $\sigma$ created this way is a uniformly random permutation of the edges of $K_n$. Further, the graphs $\G{[at/2]}$ and 
$\G{M_2}$ as defined in Section~\ref{sec:goodvertices} are independent and are drawn independently from the distribution of $G(n,at/2)$. We define five events. Let $\delta=\frac{\alpha}{4}(1+\epsilon)$.  

First, let ${\cal A}$ be the event containing those triplets $(\sigma_1,\sigma_2,\sigma_3)$ 
that $\Delta\left(\Hgraph{[at]}\right)\leq \frac{8e\alpha at}{n}$ and let  ${\cal A}_1$ be the event containing those $\sigma_1$ 
for which $\Delta\left(\Hgraph{[at/2]}\right)\leq \frac{8e\alpha at}{n}$ (note that $\Hgraph{[at/2]}$ depends only on $\sigma_1$).
Observe that ${\cal A}$ implies ${\cal A}_1$ and by Lemma \ref{lemBlockedLocalUpperbound}, 
${\cal A}$ holds a.a.s.

Furthermore, let ${\cal B}_1$ be the event containing those $\sigma_1$ such that the uniform random 
graph $\G{[at/2]}$, with $\frac{at}{2}=m=\delta n \ln n \ln\ln n$ edges, has the property that 
for any subset $B\subseteq V$, $|B| \leq \ln n$, any sequence of at most $k\leq \ln n$ pairs of vertices 
$a_1, \ldots a_k, b_1, \ldots b_k \in V\setminus B$, 
and any graph $H\subseteq K_n$ with maximum degree 
$\Delta (H) \leq \frac{m}{32 n}$ there exists $k-1$ pairwise disjoint paths 
$P_i \subseteq \G{[at/2]}-E(H)-B-\{ a_1, b_k\}$, $i=1, \ldots k-1$, of length $\leq \ln n$ each,
connecting $b_{i}$ to $a_{i+1}$. 
By Corollary \ref{cor:GnmShortPaths}, we have that ${\cal B}_1$ holds a.a.s.

Let $\sigma_1 \in {\cal A}_1 \cap {\cal B}_1$. For a set $B\in {V\choose \leq \ln n}$ and for a sequence $a_1, \ldots a_k, b_1, \ldots, b_{k} \in V\setminus B$ of 
at most $2\ln n$ distinct vertices let us denote by $D^* (B, a_1, \ldots a_k, b_1, \ldots, b_{k}) \subseteq V$ 
the union of $B\cup \{ a_2, \ldots , a_k, b_1, \ldots , b_{k-1}\}$ with the union of 
the vertex sets of the $k-1$ pairwise disjoint paths of length $\leq \ln n$
connecting $b_i$ to $a_{i+1}$, for $i=1, \ldots k-1$, in $\G{[at/2]} - 
E(\Hgraph{[at/2]}) - B -\{ a_1, b_k\} $. 
Note that these $k-1$ paths do exist since the
maximum degree $\Delta (\Hgraph{[at/2]}) \leq \frac{8e\alpha at}{n} \leq \frac{m}{32 n}$ by $\sigma_1\in {\cal A}_1$ and hence the property from ${\cal B}_1$ can be applied.
(In case the choice of the family of paths is not unique then it is selected according to an arbitrary, but fixed preference 
order.)

Let us denote by ${\cal D}^*(\sigma_1) = {\cal D}^*$ the family containing 
$D^* (B, a_1, \ldots a_k, b_1, \ldots, b_{k})$ for all choices of 
$B\in {V\choose \leq \ln n}$,  and $a_1, \ldots , a_k,$ $b_1, \ldots , b_{k}\in V\setminus B$.  Clearly, $|{\cal D}^*| < n^{3\ln n}$. 
Furthermore note that every $D^* \in {\cal D}^*$ has at most 
$\ln n + (\ln n -1) \ln n = \ln^2 n$ elements.

Let ${\cal B}_2$ be the event containing the pairs $(\sigma_1,\sigma_2)$ such that $\sigma_1 \in {\cal A}_1 \cap {\cal B}_1$ and that the
uniform random graph $\G{M_2} \sim G(n, \frac{at}{2})$, has the property 
that for every $D\in\mathcal{D}^*(\sigma_1)$ and any graph $H$ with maximum degree 
$\Delta (H) \leq \gamma_{(\ref{cor:GnmHamConnected})} \frac{m}{n}$  the graph $\G{M_2}-E(H) - D$ is Hamilton connected. 
Note that by Corollary~\ref{cor:GnmHamConnected}, the event ${\cal B}_2$ conditioned on any $\sigma_1\in {\cal A}_1 \cap {\cal B}_1$ holds with probability at least $1-e^{-\Omega(\ln n \ln \ln n)}$. Here it is crucial that, although $M_2$ depends on both $\sigma_1$ and $\sigma_2$, the graph $G\left(M_2\right)$ is independent of $\sigma_1$ by construction.
 
Finally, we let ${\cal S}$ be the event containing those triplets $(\sigma_1,\sigma_2,\sigma_3)$ such that after $t$ rounds there are disjoint paths 
$Q_1,\dots, Q_k$, $k\leq \ln{n}$, covering the set 
$B$ of $\alpha$-bad vertices and having their endpoints, and only those, among the $\alpha$-good vertices. 
Note that by Corollary \ref{cor:PathsOnB},  ${\cal S}$ holds a.a.s.

Then, formally, we have that
\begin{align*} 
& \prob{\overline{{\cal A}}} + \prob{\overline{{\cal B}_1}} + \prob{\overline{{\cal B}_2}} + \prob{\overline{{\cal S}}} =\\
& = \prob{\overline{{\cal A}}} + \prob{\overline{{\cal B}_1}} + \sum_{\sigma_1 \in {\cal A}_1\cap {\cal B}_1} \prob{ \overline{{\cal B}_2} \given \sigma_1}\prob{\sigma_1} + \prob{\overline{{\cal S}}}\\
&= o(1)+o(1)+e^{-\Omega(\ln n\ln \ln n)}\sum_{\sigma_1 \in {\cal A}_1\cap {\cal B}_1} \prob{\sigma_1}+o(1)=o(1).
\end{align*}

It remains to show that for any triplet $(\sigma_1, \sigma_2, \sigma_3)$ such that $(\sigma_1, \sigma_2, \sigma_3)\in {\cal A}\cap {\cal S}$, $\sigma_1 \in {\cal B}_1$ and 
$(\sigma_1,\sigma_2)  \in {\cal B}_2$ hold, \RM following the permutation strategy according to the $\sigma$ induced by the 
triplet $(\sigma_1, \sigma_2, \sigma_3)$ builds a Hamilton cycle against \CB playing with his fixed strategy $S$ (by the end of round $t$).

First we show that \RM's graph after $t$ rounds contains a single path of length at most $\ln^2 n$
covering the set $B$ of all $\alpha$-bad vertices. 
Indeed, ${\cal S}$ guarantees paths $Q_1,\dots,Q_k$, $k\leq \ln{n}$, partitioning $B$, and having their endpoints 
$a_1, b_1, a_2, b_2, \ldots, a_k, b_k$, and only those 
among the $\alpha$-good vertices. Recall that $\size{B}\leq \ln{n}$. 
Since $\sigma_1 \in {\cal B}_1 \cap {\cal A}_1$, there are  
$k-1$ pairwise disjoint paths $P_i \subseteq \G{[at/2]}-E(\Hgraph{[at/2]})-B-\{ a_1, b_k\}$, 
$i=1, \ldots k-1$, of length at most $\ln n$ connecting $b_i$ to $a_{i+1}$. Since only good vertices are involved in these paths, by Lemma \ref{lemGoodBlocked} the paths are indeed in \RM's graph.
The concatenation of the paths $P_i$ and $Q_j$ gives a single $a_1,b_k$-path $P$
of length at most $\ln^2 n$ covering all bad vertices.

Since $\left(\sigma_1,\sigma_2, \sigma_3\right)\in {\cal A}$ and $(\sigma_1,\sigma_2) \in {\cal B}_2$,
the graph $\G{M_2}-\Hgraph{[at]}-\left(V(P)\setminus \{ a_1, b_k\}\right) $ is Hamilton connected, and thus contains a 
Hamilton path $Q$ connecting $a_1$ and $b_k$. Note that removing $V(P)\setminus \{ a_1, b_k\}$ removes all bad vertices and thus, again by Lemma \ref{lemGoodBlocked}, $Q$ is contained in \RM's graph.

The concatenation of $Q$ and $P$ gives a Hamilton cycle.
Here we used that even though $\G{[at/2]}$ and $\G{M_2} $ might have common edges, for $Q$ we used 
only those edges of $\G{M_2}$ that are left after deleting the internal vertices of $P$.
\end{proof}

\section{Remarks and Open Problems}
In this paper we determined the sharp threshold bias of the minimum-degree-$k$, connectivity and Hamiltonicity games
in the half-random \RM vs \CB scenario. To prove that the sharp
threshold bias of the half-random $k$-connectivity game is also 
$\ln \ln n$, we can proceed as we did when deriving connectivity and
minimum-degree-$k$. Suppose there is a vertex cut $S$ of size at most
$k-1$ in \RM's graph. Note that for every bad vertex there exists $k$
vertex disjoint paths to good vertices, so the deletion of $k-1$ vertices will
not disconnect all of these paths: any bad vertex will still be
connected to a good vertex after the deletion of $S$. So it is enough
to show that the graph of \RM induced by the good vertices is not 
disconnected with the removal of $S$. This can be done similarly as 
we show the $1$-connectedness of the graph: the only difference 
is that for any $X$, $|X| \leq n/2$ we show that the number of edges 
going to $\overline{X}\setminus S$ is at least $\Omega(|X| \ln n \ln
\ln n) >0$. For this one needs to also subtract from the calculations there 
the number of edges incident to $S$, which is negligeble.

It would also be interesting to study other natural half-random games, for example non-planarity, 
non-$k$-colorability, and $k$-minor games, as well as their
half-random Avoider-Enforcer and Waiter-Client variants.

Further, it is well-known that for a fixed graph $H$ the 
threshold bias $n^{1/m(H)}$ of the random $H$-building game is coarse. 
It is unclear however whether the \RM vs \CB 
half-random $H$-creation game cannot have a sharp threshold bias, we tend 
to think it does.
Note that by~\cite{bednarska2000biased} 
we know the order of magnitude of the threshold, it is $n^{1/m_2(H)}$, 
where $m_2(H)$ is the usual maximum $2$-density of $H$.

\bibliography{mainLiterature}{}
\bibliographystyle{plain}

\section{Appendix}

\begin{proof}[Proof of Lemma~\ref{lem:EdgesToComplement}]
Fix a set $X\subset [n]$ of size $\size{X}\leq n/2$. 
Applying Theorem \ref{thm:hypergeometric} to the edge set $\left(X,\overline{X}\right)$ between $X$ and its complement, 
with $F={n\choose 2}$ and $l=m$, the probability that less than $\frac{m|X|(n-|X|)}{2{n\choose 2}}\geq \size{X}\frac{m}{2n}$ 
edges are present in $G$ between $X$ and $\overline{X}$ is at most $e^{-\size{X}\frac{m}{8n}}$.
Taking the union bound over all subsets $X$, we obtain that the failure probability is $\sum_{k=1}^{n/2} {n \choose k}e^{-k\frac{m}{8n}} = 
\sum_{k=1}^{n/2} e^{k(O(\ln n) - \Omega(\ln n \ln \ln n))} = o(1).$
\end{proof}

\begin{proof}[Proof of Lemma \ref{lem:smallSetsHaveLargeNeighborhood}]
The following claim directly implies part (a) of the lemma.\\[2mm]
\textbf{Claim:} With probability at least $1- e^{-\Omega (\ln n \ln \ln n)}$ for every vertex set $X$ of size at most $\size{X} \leq \frac{n^2}{m}$ 
there is a sequence of vertices $v_1,\dots,v_{\size{X}/2}\in X$ and disjoint sets 
$N_1,\dots,N_{\size{X}/2}\subset [n]\setminus X$ of size $\frac{m}{4n}$ each, such that for all $i=1, \ldots , \frac{|X|}{2}$ 
we have $N_i \subseteq N(v_i)$.
\begin{proof}
To prove the claim, let us first fix $X$ and write $X=\left\{s_1,\dots,s_{\size{X}}\right\}$.
We inductively define sets $M_i$: if there exist  at least $\frac{m}{4n}$ neighbors of $s_1$ in $\overline{X}$, then 
let $M_1$ be the set containing the $\frac{m}{4n}$ neighbours with the lowest index. Otherwise let $M_1=\emptyset$.  Similarly, let then $M_i$ be the set of the $\frac{m}{4n}$ neighbors of $s_i$ in $\overline{X}\setminus\left(\bigcup_{j<i}M_j\right)$ with the lowest index 
provided there are at least $\frac{m}{4n}$ such neighbors, and the empty set otherwise.

Further, call each $s_i$ a \emph{success}, if $M_i\neq\emptyset$ or $\sum_{j=1}^{i-1} d(s_j)\geq \frac{m}{2}$. 
Note that if there are less than $\frac{\size{X}}{2}$ failures, then either (1) the claim holds for $X$, or (2) $X$ has at least $\frac{m}{4}$ 
incident edges. However, the number of edges incident to $X$ has the hypergeometric distribution with parameters 
$F={n\choose 2}$, $f=\size{X}(n-\size{X}) + {|X|\choose 2}$ and $l=m$. For the expectation we have $\frac{|X|m}{n} \leq \frac{fl}{F} \leq \frac{|X|m}{2(n-1)} \leq \frac{m}{8}$, hence by Theorem \ref{thm:hypergeometric} 
the probability of (2) is at most $e^{-m\size{X}/3n}$. 
We show now that with high probability there are less than $\frac{|X|}{2}$ failures.

We go through the $s_i$ in increasing order, and determine the probability of a failure, conditioned under the exact sets of neighbors of the $s_1,\dots,s_{i-1}$. So, fix an $i\leq \size{X}$ and condition on the event that for $j<i$, the neighborhood of $s_j$ is $N_G(s_j)=B_j$ for some fixed sets $B_j$.\\
Now if $\sum_{j<i}\size{B_j}\geq \frac{m}{2}$, then $s_i$ is a guaranteed success. Otherwise, there are $l\geq \frac{m}{2}$ edges left to place in $G$, and 
$F\leq {n\choose 2}$ potential edges to choose from. The ``good'' edges are all the edges from $s_i$ to 
$\overline{X}\setminus \bigcup_{j<i}M_j$, there are at least $f\geq n-|X|-(i-1)\frac{m}{4n}\geq \frac{n}{2}$ of them. 
For the expectation we have $\frac{fl}{F} \geq \frac{m}{2n}$, so by Theorem \ref{thm:hypergeometric} the probability that $s_i$ is a failure, that is, 
that less than $\frac{m}{4n}$ of the good edges are realized, is at most $e^{-m/16n}$.

Thus, by the union bound the probability that there are $\frac{|X|}{2}$ elements of $X$ that are failures is at most
$${\size{X}\choose |X|/2}\left(e^{-m/16n}\right)^{|X|/2}
=e^{-\Omega(m\size{X}/n)}.$$
This means that the probability that the claim does not hold for $X$ is at most $e^{-\Omega(m\size{X}/n)} + e^{-m|X|/3n}$.

Now taking the union bound over all $X$, $|X|\leq \frac{n^2}{m}$, the probability that the event of the claim does not hold is at most
$$\sum_{x=1}^{n/\ln^2{n}}e^{x\ln{n}}e^{-\Omega(mx/n)}=e^{-\Omega\left(\ln n\ln\ln n\right)}.$$
\end{proof}

We prove that part (b) holds with probability at least $1-e^{-n/10}$. 
Let us fix sets $X$ and $N$. The number of edges between $X$ and $[n]\setminus(X\cup N)$ has the hypergeometric distribution with parameters $F={n\choose 2}$, $f=\size{X}\size{[n]\setminus(X\cup N)}\geq \size{X}\frac{n}{4}$ and $l=m$. By Theorem \ref{thm:hypergeometric}, we have that the probability that there are less than 
$\size{X}\frac{m}{8n} \leq \frac{fl}{2F}$ edges between $X$ and $[n]\setminus(X\cup N)$ is at most $e^{-fl/8F}\leq e^{-\size{X}\frac{m}{32n}}\leq e^{-2n}$. Since there 
are at most $2^n2^n$ pairs of sets $X$ and $N$, the probability that the statement fails is at most $e^{\left(2\ln{2}-2\right)n} \leq e^{-n/10}$.

Finally we show that part (c) holds with probability at least $1-e^{-m/129}$.
Let us fix disjoint vertex sets $X,Y\subset [n]$ of size at least $\size{X},\size{Y}\geq \frac{n}{4}$. The number of edges between $X$ and $Y$ follows the hypergeometric distribution with parameters $F={n\choose 2}$, $f=\size{X}\size{Y}\geq \frac{|X|n}{4}$ and $l=m$.  By Theorem \ref{thm:hypergeometric}, we have that the probability that there are less than $\size{X}\frac{m}{8 n} \leq \frac{fl}{2F}$ edges between $X$ and $Y$ is at most 
$e^{-fl/8F}\leq e^{-\size{X}m/32n}\leq e^{-m/128}$. Since there are at most $2^n2^n\leq e^{o(m)}$ pairs of sets $X$ and $Y$, the claim follows by taking the union bound over all such $X$ and $Y$.
\end{proof}

\end{document}